\documentclass[english]{article}
\usepackage[T1]{fontenc}
\usepackage[latin9]{inputenc}
\usepackage{geometry}
\usepackage{amsmath}
\usepackage{amssymb}
\usepackage{amsthm}
\usepackage{babel}
\usepackage{graphicx}
\usepackage{setspace}
\usepackage{tikz}
\usepackage{comment}
\usepackage{hyperref}
\usepackage{cite}
\usepackage{cleveref}
\newcommand{\sofour}{$\mathfrak{so}_4$}
\newcommand{\sotwon}{$\mathfrak{so}_{2n}$}
\newcommand{\half}{\frac{1}{2}}
\newcommand{\rhoCC}{\rho_{\mathbb{C}^{2n}\otimes\mathbb{C}^{2n}}}
\newcommand\cev[1]{\overleftarrow{#1}}
\DeclareMathOperator{\Tr}{Tr}
\newtheorem{theorem}{Theorem}[section]
\newtheorem{proposition}[theorem]{Proposition}
\newtheorem{lemma}[theorem]{Lemma}

\theoremstyle{remark}
\newtheorem{remark}{Remark}
\newtheorem{example}{Example}
\newtheorem{definition}[theorem]{Definition}
\usepackage{multicol}

\title{Interacting particle systems with type $D$ symmetry and duality }
\author{Jeffrey Kuan, Mark Landry, Andrew Lin, Andrew Park, Zhengye Zhou}

\begin{document}

\maketitle
\date{}

\begin{abstract}
We construct a two--class asymmetric interacting particle system with $\mathcal{U}_q(\mathfrak{so}_6)$ or $\mathcal{U}_q(\mathfrak{so}_8)$  symmetry, in which up to two particles may occupy a site if the two particles have different class. The particles exhibit a drift, but there is no preference given between first--class and second--class particles. The quantum group symmetry leads to reversible measures and a self--duality for the particle system. Additionally, a new method is developed to construct a symmetric interacting particle system from the Casimir element of $\mathfrak{so}_{2n}$.
\end{abstract}

Recent research \cite{BS2},\cite{BS},\cite{KuanJPhysA},\cite{KIMRN},\cite{BS3} has shown that multi--class (also known as multi--species, or colored) particle systems can be constructed from the representation theory of the Drinfeld--Jimbo quantum group $\mathcal{U}_q(\mathfrak{g})$, where $\mathfrak{g}$ is a Lie algebra of rank at least two. The generator of the particle system is constructed from a central element of $\mathcal{U}_q(\mathfrak{g})$. The symmetry of the generator with the quantum group leads to a self--duality for the particle system. The general framework connecting quantum groups to duality was laid out in \cite{CGRS}; see also the previous papers \cite{Sch97} and \cite{IS}, and also \cite{CGRS2}.

This paper will construct central elements of $\mathcal{U}_q(\mathfrak{so}_6)$ and $\mathcal{U}_q(\mathfrak{so}_8)$ using a general method articulated in \cite{KuanJPhysA}, and then constructs a two--class asymmetric interacting particle system from these central elements. The particle system is a generalization of the asymmetric simple exclusion process (ASEP), where two particles may occupy the same site if they have different class. A self--duality is then proven using the quantum group symmetry.

When constructing the particle systems, certain states corresponding to ``negative probabilities'' are discarded. It is natural to aim to construct a process with no ``discarded states.'' By allowing for more flexibility in the construction, such a process is constructed from the Casimir element of $\mathcal{U}(\mathfrak{so}_{2n})$ for generic values of $n$.

The paper is outlined as follows: section \ref{First} provides the relevant background and states the algebraic results (a central element of $\mathcal{U}_q(\mathfrak{so}_6)$ and $\mathcal{U}_q(\mathfrak{so}_8)$) and probabilistic results (reversible measures and self--duality of a two-class particle system). Section \ref{Second} provides the proofs of the results stated in section \ref{First}. Section \ref{Third} constructs a symmetric interacting particle system from the Casimir element of $\mathcal{U}(\mathfrak{so}_{2n})$. 

\textbf{Acknowledgements}
This research was conducted as part of a REU at Texas A\&M University, funded by the National Science Foundation  (DMS--1757872).

\section{Definitions and Results}\label{First}

\subsection{Definitions}
\subsubsection{Algebraic Definitions}

The Lie algebra \sotwon\  consists of matrices of the form:
\[
\mathfrak{so}_{2n}(\mathbb{C})=\bigg\{
\begin{pmatrix}
A & B\\
C & D
\end{pmatrix}
\bigg{|}A,B,C,D\in\mathbb{C}^{n\times n}, A=-D^T,B=B^T,C=C^T\bigg\}
\]
Let $E_{i,j}$ be the matrix with 1 in the $(i,j)$ entry and 0 elsewhere. A Cartan subalgebra $\mathfrak{h}$ of \sotwon \ is the subalgebra of diagonal matrices.

Let $L_i$ be defined as the linear map on the Cartan subalgebra where $L_i$ maps $H$ to its $i$-th diagonal entry. Then the roots and root vectors of \sotwon \ are defined for $1 \leq i,j\leq n$ as:
\begin{itemize}
    \item $X_{i,j}=E_{i,j}-E_{n+j,n+i}$ has root $L_i-L_j$
    \item $Y_{i,j}=E_{i,n+j}-E_{j,n+i}$ has root $L_i+L_j$
    \item $Z_{i,j}=E_{n+i,j}-E_{n+j,i}$ has root $-L_i-L_j$
\end{itemize}

The positive roots are $R^+=\{L_i+L_j\}_{i<j}\cup \{L_i-L_j\}_{i<j}$, and the simple roots are $L_1-L_2,L_2-L_3,\dots,L_{n-1}-L_{n},L_{n-1}+L_n$. Additionally, in \sotwon, the Killing form can be defined using $B(X,Y)=(2n-2)\Tr(XY)$.

Below, we list a few key algebraic properties associated to this Lie algebra:

\begin{itemize}
\item The rank of the Lie algebra $\mathfrak{so}_{2n}$ is $n$. Specifically, if $L_i$ denotes the linear operator taking a matrix $M$ to its diagonal entry $M_{i,i}$, then the set of roots is $\{L_i - L_j: 1 \le i \ne j \le n\}$, and the positive simple roots are $\alpha_i = L_i - L_{i+1}$ for $1 \le i \le n-1$ and $\alpha_n = L_{n-1} + L_n$.
\item The \textit{Cartan subalgebra} $\mathfrak{h}$ of $\mathfrak{so}_{2n}$ is the subalgebra of diagonal matrices spanned (as a vector space) by the matrices $H_n = E_{n-1, n-1} + E_{n, n} - E_{2n-1, 2n-1} - E_{2n, 2n}$ and, for $1 \le i \le n-1$),
\[
    H_i = E_{i,i} - E_{i+1,i+1} - E_{n+i,n+i} + E_{n+i+1, n+i+1},
\]
where $E_{a, b}$ denotes the matrix with a $1$ in the $(a, b)$th entry and $0$s elsewhere. In other words, $H_i$ takes its first $n$ diagonal entries from the $i$th positive simple root and then flips the signs for the next $n$ entries (so that $A = -D^T$).
\item The corresponding Dynkin diagram is $D_n$. In particular, all nonzero off-diagonal entries of the Cartan matrix are $-1$: 
\[
    a_{ij} = \begin{cases} 2 & i = j \\ -1 & \{i, j\} = \{n-2, n\} \text{ or } \{k, k+1\}, 1 \le k \le n-2 \\ 0 & \text{otherwise.} \end{cases}
\]
\item In the \textbf{fundamental representation} of $\mathfrak{so}_{2n}$, all elements of the Lie algebra act on a $2n$-dimensional vector space. Certain elements correspond to simple roots and will be referred to later in the paper: for all $1 \le i \le n-1$, we have
\[ 
    E_i = E_{i,i+1} - E_{n+i+1,n+i}, \quad E_n = E_{n-1, 2n} - E_{n,2n-1},
\]
\[
    F_i = E_{i+1,i} - E_{n+i,n+i+1}, \quad F_n = E_{2n-1,n} - E_{2n,n-1}. 
\]
The weights of the fundamental representation are $\pm L_i$ (for $1 \le i \le n$).
\end{itemize}

\begin{definition}
The \textbf{quantum group} $U_q(\mathfrak{so}_{2n})$ is the algebra generated by $\{E_i, F_i, q^{H_i}: 1 \le i \le n\}$. These generators satisfy the relations
\[
    [E_i, F_i] = \frac{q^{H_i} - q^{-H_i}}{q-q^{-1}}, \quad q^{H_i}E_j = q^{(\alpha_i, \alpha_j)} E_j q^{H_i}, \quad q^{H_i} F_j = q^{-(\alpha_i, \alpha_j)} F_j q^{H_i}, 
\]
as well as the Serre relation for all $(i, j)$ with $a_{ij} = -1$ in the Cartan matrix 
\[
    E_i^2E_j + E_jE_i^2 = (1+q) E_iE_jE_i, \quad F_i^2F_j + F_jF_i^2 = (1+q) F_iF_jF_i,
\]
and all other pairs of elements commuting. Finally, the coproducts of the generators are
\[
    \Delta (E_i) = E_i \otimes 1 + q^{H_i} \otimes E_i, \quad \Delta(F_i) = 1 \otimes F_i + F_i \otimes q^{-H_i}, \quad \Delta(q^{H_i}) = q^{H_i} \otimes q^{H_i}.
\] 
We will also write $K_i$ for $q^{H_i}$.
\end{definition}

A general definition for the quantum group of a Lie algebra $\mathfrak{g}$ can be found in \cite{jantzen}; this quantum group serves as a deformation of the ordinary universal enveloping algebra $U(\mathfrak{g})$.

Constructing a nondegenerate bilinear form often allows us to study relationships between elements of a vector space, and Chapter 6 of \cite{jantzen} introduces a similar idea. Recall that the \textbf{Borel subalgebras} $\mathfrak{b}\pm$ are the Lie subalgebras generated by $\{E_i, H_i\}$ and $\{F_i, H_i\}$, respectively. Let $U_q(\mathfrak{b}\pm)$ denote the corresponding subalgebras of the quantum group generated by the Borel subalgebras (replacing $H_i$ with $q^{H_i}$), and let $\langle \, , \, \rangle$ be the bilinear pairing $U_q(\mathfrak{b}-) \times U_q(\mathfrak{b}+) \to \mathbb{Q}(q)$ such that for any linear combinations $\alpha, \beta$ of the $\alpha_i$s, we have
\[
    \langle q^{H_\alpha}, q^{H_\beta} \rangle = q^{-(\alpha \cdot \beta)} \text{ and } \langle F_i, E_j \rangle = -\delta_{ij}(q-q^{-1})^{-1},
\]
and all other pairings between generators are zero. Furthermore, the pairing can be computed for products via
\[
    \langle y, xx' \rangle = \langle \Delta(y), x' \otimes x \rangle, \quad \langle yy', x \rangle = \langle y \otimes y', \Delta(x) \rangle,
\]
where the coproduct $\Delta$ satisfies $\Delta(ab) = \Delta(a)\Delta(b)$ and has explicit form given in the definition above. (Here, $\langle x_1 \otimes x_2, y_1 \otimes y_2 \rangle$ is defined to be $\langle x_1, y_1 \rangle \langle x_2, y_2 \rangle$.) 

\subsubsection{Probabilistic definitions}
Recall that two Markov processes $X(t)$ and $Y(t)$ on state spaces $\mathcal{X}$ and $\mathcal{Y}$ are dual to the function $D(x,y)$ on $\mathcal{X} \times \mathcal{Y}$ if
$$
\mathbb{E}_x[D(X(t),y)] = \mathbb{E}_y[D(x,Y(t))]
$$
for all $x\in \mathcal{X}, y \in \mathcal{Y}$ and $t \geq 0$. If $\mathcal{X}$ and $\mathcal{Y}$ are discrete and $L_X,L_Y$ denote the generators of $X(t),Y(t)$, then an equivalent description is
$$
L_XD = DL_Y^T
$$
where $^T$ denotes transpose. If $L_X=L_Y$ then we say that the Markov process is self--dual.

\subsection{Results}

\subsubsection{Algebraic results}

To make the expression more readable, we adopt some shortcuts: set $r = q - \frac{1}{q}$, and let $E_{x_1x_2 \cdots x_n}$ denote $E_{x_1}E_{x_2} \cdots E_{x_n}$ (and similar for $F$s). (For example, $r^2 E_{21}F_3$ would denote $(q - \frac{1}{q})^2 E_2E_1F_3$.)

\begin{theorem}\label{Center}
(a) The following element of the quantum group $U_q(\mathfrak{so}_6)$ is central: 
\begin{footnotesize}
\begin{equation*}
\begin{split}
q^{-4-2H_1-H_2-H_3} + q^{-2-H_2-H_3} + q^{H_2-H_3} + q^{H_3-H_2} + q^{2+H_2+H_3} + q^{4+2H_1+H_2+H_3} + \tfrac{r^2}{q^3} F_1q^{-H_1-H_2-H_3}E_1 \\
+ \tfrac{r^2}{q} F_2q^{-H_3}E_2 - \tfrac{r^2}{q} F_3q^{-H_2}E_3 + r^2qF_2q^{H_3}E_2 - r^2qF_3q^{H_2}E_3 + r^2q^3F_1q^{H_1+H_2+H_3}E_1 \\+ \tfrac{r^2}{q^3} (qF_{12} - F_{21})q^{-H_1-H_3}(qE_{21} - E_{12}) - \tfrac{r^2}{q^3} (qF_{13} - F_{31}) q^{-H_1-H_2}(qE_{31} - E_{13}) \\+ r^2q (qF_{21} - F_{12})q^{H_1 + H_3}(qE_{12} - E_{21}) - r^2q(qF_{31} - F_{13})q^{H_1+H_2}(qE_{13} - E_{31}) 
\\ - \tfrac{r^2}{q^3}(q^2F_{123} - qF_{213} - qF_{312} + F_{231})q^{-H_1} (q^2E_{231} - qE_{312} - qE_{213} + E_{123})
\\ -\tfrac{r^2}{q} (q^2F_{231} - qF_{312} - qF_{213} + F_{123})q^{H_1}(q^2E_{123} - qE_{213} - qE_{312} + E_{231})
\\ -\tfrac{r^4}{q^2} ((q^2+1)F_{1231} - qF_{1312} - qF_{2131})((q^2+1)E_{1231} - qE_{1312} - qE_{2131})
\\ - r^4F_2F_3E_2E_3.
\end{split}
\end{equation*}
\end{footnotesize}
This element acts as $q^6 + q^2 + 2 + q^{-2} + q^{-6}$ times the identity matrix in the fundamental representation of $\mathfrak{so}_6$.

(b) The following element of the quantum group $U_q(\mathfrak{so}_8)$ is central: 
\begin{footnotesize}
\begin{equation*}
\begin{split}
q^{-6-2H_1-2H_2-H_3-H_4} + q^{-4-2H_2-H_3-H_4} + q^{-2-H_3-H_4} + q^{H_3-H_4} \\
+q^{H_4 - H_3} + q^{2 + H_3+H_4} + q^{4 + 2H_2 + H_3 + H_4} + q^{6 + 2H_1 + 2H_2 + H_3 + H_4} \\
+ \tfrac{r^2}{q^5} F_1q^{-H_1-2H_2-H_3-H_4}E_1 + \tfrac{r^2}{q^5} (qF_{12} - F_{21}) q^{-H_1-H_2-H_3-H_4} (qE_{21} - E_{12}) \\
 + \tfrac{r^2}{q^5} (q^2F_{123} - qF_{132} - qF_{213} + F_{321}) q^{-H_1-H_2-H_4} (q^2E_{321} - qE_{213} - qE_{132} + E_{123}) \\
 - \tfrac{r^2}{q^5} (q^2F_{124} - qF_{142} - qF_{241} + F_{421} q^{-H_1 -H_2-H_3} (q^2E_{421} - qE_{241} - qE_{142} + E_{124}) \\
 - \tfrac{r^2}{q^5} \boxed{A_1} q^{-H_1 - H_2} \boxed{A_{4}} - \tfrac{r^2}{q^5} \boxed{A_5} q^{-H_1} \boxed{A_8} - \tfrac{r^4}{q^4} \boxed{A_9}\,\boxed{A_{10}} \\
 + \tfrac{r^2}{q^3} F_2q^{-H_2-H_3-H_4}E_2 + \tfrac{r^2}{q^3} (qF_{23} - F_{32}) q^{-H_2-H_4} (qE_{32} - E_{23}) - \tfrac{r^2}{q^3} (qF_{24} - F_{42}) q^{-H_2-H_3} (qE_{42} - E_{24}) \\
 - \tfrac{r^2}{q^3} (q^2F_{234} - qF_{324} - qF_{423} + F_{432}) q^{-H_2} (q^2E_{432} - qE_{324} - qE_{423} + E_{234}) \\ 
 -\tfrac{r^4}{q^2} ((q^2+1)F_{2342} - qF_{3242} - qF_{2423})((q^2+1)E_{2342} -
 qE_{3242} - qE_{2423}) \\
 - \tfrac{r^2}{q^3} \boxed{A_7} q^{H_1} \boxed{A_6} + \tfrac{r^2}{q}F_3q^{-H_4}E_3 - \tfrac{r^2}{q}F_4q^{-H_3}E_4 \\
 - r^4F_3F_4E_4E_3 - \tfrac{r^2}{q}(q^2F_{432} - qF_{324} - qF_{423} + F_{234}) q^{H_2} (q^2E_{234} - qE_{324} - qE_{423} + E_{432}) \\
 - \tfrac{r^2}{q}\boxed{A_{3}} q^{H_1+H_2}\boxed{A_2} - r^2qF_4q^{H_3}E_4 - r^2q (qF_{42} - F_{24})q^{H_2+H_3} (qE_{24} - E_{42}) \\
 - r^2q(q^2F_{421} - qF_{241} - qF_{142} + F_{124}) q^{H_1+H_2+H_3}(q^2E_{124} - qE_{142} - qE_{241} + E_{421}) \\
 + r^2qF_3q^{H_4}E_3 + r^2q(qF_{32} - F_{32}) q^{H_2+H_4}(qE_{23} - E_{32}) \\
 + r^2q(q^2F_{321} - qF_{213} - qF_{132} + F_{123}) q^{H_1 + H_2 + H_4} (q^2E_{123} - qE_{132} - qE_{213} + E_{321}) \\
 + r^2q^3F_2q^{H_2+H_3+H_4}E_2 + r^2q^3(qF_{21} - F_{12})q^{H_1+H_2+H_3+H_4}(qE_{12}-E_{21}) + r^2q^5F_1q^{H_1+2H_2+H_3+H_4}E_1.
\end{split}
\end{equation*}
\end{footnotesize}
where
\[
    A_1 = q^3 F_{1234} - q^2 F_{2314} - q^2 F_{3124} - q^2 F_{1423} + q F_{4213} + q F_{3241} + q F_{4132} - F_{4321},
\]
\[
    A_2 = q^3 E_{1234} - q^2 E_{2314} - q^2E_{3124} - q^2 E_{1423} + qE_{4213} + q E_{3241} + q E_{4132} - E_{4321},
\]
\[
    A_3 = q^3 F_{4321} - q^2 F_{4132} - q^2 F_{4213} - q^2 F_{3241} + q F_{1423} + q F_{2314} + q F_{3124} - F_{1234},
\]
\[
    A_4 = q^3 E_{4321}  - q^2 E_{4132} - q^2 E_{4213} - q^2 E_{3241} + q E_{1423} + q E_{2314} + q E_{3124} - E_{1234}
\]
and
\begin{equation*}
\begin{split}
    A_5 = q^4 F_{12342} + F_{23421} + q^2 F_{42132} + q^2 F_{24123} - q F_{23214} + q^2 F_{23124} - (q^3 + q) F_{23142}\\ - \tfrac{q^4}{q^2+1} F_{12243}  - q^3 F_{13242} + q^2 F_{32142} + \tfrac{q^4 - q^2}{q^2 + 1} F_{42213} - q^3 F_{42123} - \tfrac{q^2}{q^2 +1} F_{34221}, \\
    A_6 = q^4 E_{12342} + E_{23421} + q^2 E_{42132} + q^2 E_{24123} - q E_{23214} + q^2 E_{23124} - (q^3 + q) E_{23142} \\- \tfrac{q^4}{q^2 + 1} E_{12243} - q^3 E_{13242} + q^2 E_{32142} + \tfrac{q^4 - q^2}{q^2 + 1} E_{42213} - q^3 E_{42123} - \tfrac{q^2}{q^2 +1} E_{34221}, \\
    A_7 = F_{12342} + q^4 F_{23421} + q^2 F_{42132} + q^2 F_{24123} - q^3 F_{23214} + q^2  F_{23124} - (q^3 + q) F_{23142} \\ - \tfrac{q^2}{q^2+1} F_{12243} - qF_{13242} + q^2 F_{32142} - \tfrac{q^4-q^2}{q^2+1}F_{42213} - q F_{42123} - \tfrac{q^4}{q^2+1}F_{34221}, \\
    A_8 = E_{12342} + q^4 E_{23421} + q^2 E_{42132} + q^2 E_{24123} - q^3 E_{23214} + q^2 E_{23124} - (q^3 + q) E_{23142} \\ - \tfrac{q^2}{q^2 + 1} E_{12243} - q E_{13242} + q^2 E_{32142} - \tfrac{q^4-q^2}{q^2+1} E_{42213} - q E_{42123} - \tfrac{q^4}{q^2+1}E_{34221}
\end{split}
\end{equation*}
and
\begin{equation*}
\begin{split}
A_9 = (-q^3 - q)F_{121342} -\tfrac{q^4}{(q^2 + 1)^2} F_{223141} + q^2 F_{143122} - q^2 F_{122341} - \tfrac{q^2(q^4+q^2+1)}{(q^2+1)^2} F_{412231} + q^2 F_{241312} \\ - \tfrac{q^3}{q^2+1}F_{131242}   - \tfrac{q^3}{q^2+1} F_{421231} + \tfrac{q^3}{q^2+1}F_{232141} - (q^3 + q) F_{413212} - q^2 F_{312241} + q^2 F_{132412} \\ + \tfrac{q^3}{q^2+1} F_{114232} + q^2 F_{421321} + q^2 F_{123124} + q^2 F_{214231} - (q^3 + q) F_{231421} + (q^4 + q^2 + 1) F_{124321}, \\ 
A_{10} = (-q^3 - q)E_{121342} -\tfrac{q^4}{(q^2 + 1)^2} E_{223141} + q^2 E_{143122} - q^2 E_{122341} - \tfrac{q^2(q^4+q^2+1)}{(q^2+1)^2} E_{412231} + q^2 E_{241312} \\ - \tfrac{q^3}{q^2+1}E_{131242}   - \tfrac{q^3}{q^2+1} E_{421231} + \tfrac{q^3}{q^2+1}E_{232141} - (q^3 + q) E_{413212} - q^2 E_{312241} + q^2 E_{132412} \\ + \tfrac{q^3}{q^2+1} E_{114232} + q^2 E_{421321} + q^2 E_{123124} + q^2 E_{214231} - (q^3 + q) E_{231421} + (q^4 + q^2 + 1) E_{124321}.
\end{split}
\end{equation*}

This element acts as $q^8 + q^4 + q^2 + 2 + q^{-2} + q^{-4} + q^{-8}$ times the identity matrix in the fundamental representation of $\mathfrak{so}_8$.
\end{theorem}

\subsubsection{Probabilistic results}

Let us first describe the models. There are two classes of particles, which are called first--class and second--class. At most two particles may occupy a site, and two particles may occupy the same lattice site if they have different class. Each particle (regardless of class) evolves as an ASEP with drift to the right (respectively left) for $0<q<1$ (respectively for $q>1$), and each hole (i.e. absence of a particle) evolves as an ASEP in the opposite direction:

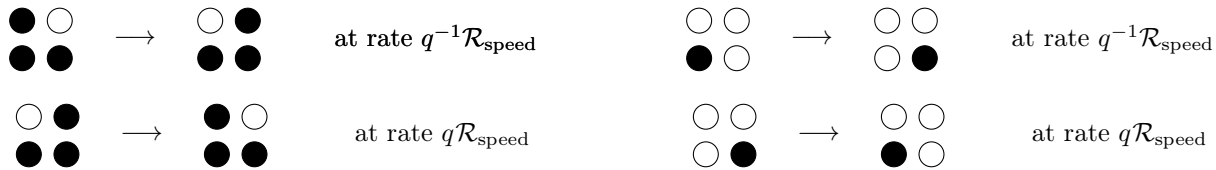
\begin{center}
\begin{tikzpicture}[scale=1]
\node[circle,draw,fill=black] at (-4,0){};
    \node[circle,draw,fill=black] at (-4,0.5){};
    \node[circle,draw,fill=black] at (-3.5,0){};
    \node[circle,draw,fill=white] at (-3.5,0.5){};
    \node[] at (-2.5,0.25){$\longrightarrow$};
    \node[circle,draw,fill=black] at (-1.5,0){};
    \node[circle,draw,fill=white] at (-1.5,.5){};
    \node[circle,draw,fill=black] at (-1,0.5){};
    \node[circle,draw,fill=black] at (-1,0){};
    \node at (1.5,0.25){at rate $q^{-1}\mathcal{R}_{\text{speed}}$};
\node[circle,draw,fill=black] at (5,0){};
    \node[circle,draw,fill=white] at (9-4,0.5){};
    \node[circle,draw,fill=white] at (9-3.5,0){};
    \node[circle,draw,fill=white] at (9-3.5,0.5){};
    \node[] at (9-2.5,0.25){$\longrightarrow$};
    \node[circle,draw,fill=white] at (9-1.5,0.5){};
    \node[circle,draw,fill=white] at (9-1.5,0){};
    \node[circle,draw,fill=white] at (9-1,0.5){};
    \node[circle,draw,fill=black] at (9-1,0){};
    \node at (1.5,0.25){at rate $q^{-1}\mathcal{R}_{\text{speed}}$};
    \node at (9+1.5,0.25){at rate $q^{-1}\mathcal{R}_{\text{speed}}$};
\end{tikzpicture}
\end{center}
\begin{center}
\begin{tikzpicture}[scale=1]
\node[circle,draw,fill=black] at (-4,0){};
    \node[circle,draw,fill=white] at (-4,0.5){};
    \node[circle,draw,fill=black] at (-3.5,0){};
    \node[circle,draw,fill=black] at (-3.5,0.5){};
    \node[] at (-2.5,0.25){$\longrightarrow$};
    \node[circle,draw,fill=black] at (-1.5,0){};
    \node[circle,draw,fill=black] at (-1.5,.5){};
    \node[circle,draw,fill=white] at (-1,0.5){};
    \node[circle,draw,fill=black] at (-1,0){};
    \node at (1.5,0.25){at rate $q^{}\mathcal{R}_{\text{speed}}$};
\node[circle,draw,fill=white] at (5,0){};
    \node[circle,draw,fill=white] at (9-4,0.5){};
    \node[circle,draw,fill=black] at (9-3.5,0){};
    \node[circle,draw,fill=white] at (9-3.5,0.5){};
    \node[] at (9-2.5,0.25){$\longrightarrow$};
    \node[circle,draw,fill=white] at (9-1.5,0.5){};
    \node[circle,draw,fill=black] at (9-1.5,0){};
    \node[circle,draw,fill=white] at (9-1,0.5){};
    \node[circle,draw,fill=white] at (9-1,0){};
    \node at (9+1.5,0.25){at rate $q^{}\mathcal{R}_{\text{speed}}$};
\end{tikzpicture}
\end{center}

Non--trivial local interactions occur when only a first-class particle and second--class particle occupy two adjacent sites. The interaction can be described in terms of ``clocks and coins.'' We describe two types of interactions.

The clocks are located at the bonds between adjacent vertex sites. There are five different clocks which specify the update of the particle system. These five clocks are labeled SWAP, STICK--RIGHT, STICK--LEFT, TWIST--RIGHT and TWIST--LEFT. Supposing the clocks are located at the bond between $x$ and $x+1$, these clocks are described as follows:

\begin{itemize}
\item
The clock SWAP swaps the particles at $x$ with the particles at $x+1$. It has rate $\mathcal{R}_{\text{SWAP}}(q)$ if both $x$ and $x+1$ have one particle. It has rate $q^{-2}\mathcal{R}_{\text{SWAP}}(q)$ if site $x$ has two particles and $x+1$ has none, and rate $q^{2}\mathcal{R}_{\text{SWAP}}(q)$ if site $x$ has no particles and $x+1$ has two. 

\item
The clock STICK takes a particle located at $x$ and a particle located at $x+1$ and sticks them together. STICK--RIGHT places them at $x+1$ and STICK--LEFT places them at $x$. STICK--RIGHT has rate $\mathcal{R}_{\text{STICK}}(q^{-1})$ if the first--class particle was initially on the right, and has rate $q^{2\delta}{R}_{\text{STICK}}(q^{-1})$ if the first--class particle was initially on the left. STICK--LEFT has rate $q^{-2\delta}\mathcal{R}_{\text{STICK}}(q) $ if the first--class particle was initially on the right and rate $\mathcal{R}_{\text{STICK}}(q)$ if the first--class particle was initially on the left. 

\item The clock TWIST takes two particles at a site and splits them apart. TWIST--LEFT takes two particles at $x+1$ and places them at $x$ and $x+1$. It has rate $(1+q^{2\delta})\mathcal{R}_{\text{TWIST}}(q)$ and places the first--class particle at $x$ with probability $1/(1+q^{2\delta})$ and the second--class particle at $x$ with probability $q^{2\delta}/(1+q^{2\delta})$. Similarly, TWIST--RIGHT takes two particles at $x$ and places them at $x$ and $x+1$. It has rate $(1+q^{-2\delta})\mathcal{R}_{\text{TWIST}}(q^{-1})$ and again places the first--class particle at $x+1$ with probability $1/(1+q^{2\delta})$ and the second--class particle at $x+1$ with probability $q^{2\delta}/(1+q^{2\delta})$. 
\end{itemize}

Note that for $0<q<1$, the TWIST clocks clearly preference the first--class particles to jump right over the second--class particles. On the other hand, for $0<q<1$ the STICK clocks are faster when the first--class particle is to the right of the second--class particle, so the STICK clocks preference the second--class particles to jump right. If $\delta=0$, then there is no preference between first and second class particles.  

The \textit{type D ASEP} with parameters $(n,\delta)$ is the interacting particle system described above with jump rates
\begin{align*}
\mathcal{R}_{\text{speed}} (q) &= q^{2n-1}+q^{-2n+1}, \\
\mathcal{R}_{\text{SWAP}} (q)&= (q^{n-1}-q^{-n+1})^2,\\
\mathcal{R}_{\text{TWIST}} (q)&= 2q^2 - q^{-2(n-2)}+q^{-2(n-1)} \\
\mathcal{R}_{\text{STICK}} (q)&= q^{2n} - q^{2(n-1)} + 2
\end{align*}
and the parameter $\delta$ occuring in the asymmetry of the STICK/TWIST clocks. 

This paper will consider five models in total, with parameters $(3,0),(3,1),(4,0),(4,1),(4,2)$.






Here is some notation that will be used in the statements of the results about the models. A particle configuration can be written as a function $\eta$ where $\eta: \{1,\ldots,N\}\rightarrow \{\emptyset, 1 , 2, 12\}$ is defined by setting $\eta(x)$ to be the particles located at site $x$. Let $A_1(\eta) \subseteq \{1,\ldots,N\}$ denote the locations of the first--class particles, and similarly define $A_2(\eta)$. These two sets are in general not disjoint.
Let $\cev{N}_x^1(\eta)$ denote the number of first--class particles to the left of $x$. More precisely, $\cev{N}_x^1(\eta)= \left| \{1,\ldots,x-1\} \cap A_1(\eta) \right|.$

The first result concerns reversible measures:
\begin{proposition}\label{rev} For all five processes, the measure on $\eta$ given by 
$$
 G^2(\eta) = \prod_{x \in A_1(\eta)} q^{-2x} \prod_{x\in A_2(\eta)} q^{-2x}
$$
is reversible. 
\end{proposition}

\begin{remark}
Since the reversible measures are invariant under switching $A_1(\eta),A_2(\eta)$, this can be interpreted as the statement that the particle system does not preference first--class particles over second--class particles. In a sense, this means that the STICK and TWIST clocks cancel each other out. For $\delta\neq 0$ this is not immediately clear.
\end{remark}

\begin{theorem}\label{dual}
The type $D$ ASEP with parameters $(2,1)$ or $(3,1)$ or $(3,2)$ is self--dual with respect to the function
$$
D(\eta,\xi) = 1_{\{ A_2(\xi)=A_2(\eta), A_1(\xi) \subseteq A_1(\eta)\}} \prod_{x \in A_1(\xi)} q^{2x - 2\cev{N}_x^1(\eta)} \prod_{x \in A_2(\xi)} q^{2x}.
$$
\end{theorem}

\begin{remark}
When $A_2(\xi)=A_2(\eta)=\emptyset$, the particle system is the usual ASEP, and the duality is Sch\"{u}tz's duality function \cite{Sch97}.
\end{remark}

\begin{example}
Suppose $N=2$ and write a particle system as $\eta = [\eta(1);\eta(2)]$. Then
\begin{align*}
LD([12;1],[2;1]) &= L( [12;1], [12;1] ) D( [12;1],[2;1]) + L( [12;1], [1;12]) D([1;12], [2;1])\\
&= -q^{-1} \mathcal{R}_{\mathrm{speed}}q^{4-2}q^2 +0 
\end{align*}
and
\begin{align*}
DL^T([12;1],[2;1]) &= D( [12;1],[2;1]) L([2;1],[2;1]) + D( [12;1], [12;\emptyset])L ( [2;1],[12;\emptyset])\\
&=q^{4-2}q^2 \left( -\mathcal{R}_{\mathrm{SWAP}} - q^{-2\delta}\mathcal{R}_{\mathrm{STICK}}(q) - \mathcal{R}_{\mathrm{STICK}}(q^{-1})\right) + q^4 (q^{-2\delta} \mathcal{R}_{\mathrm{STICK}}(q)).
\end{align*}
In order for duality to hold, we must have
$$
q^{-1} \mathcal{R}_{\mathrm{speed}} = \mathcal{R}_{\mathrm{SWAP}} + \mathcal{R}_{\mathrm{STICK}}(q^{-1}),
$$
which holds for the given jump rates. 
\end{example}

\begin{example}
Similarly,
$$
LD( [12;1],[\emptyset;12]) =  q^{-1}\mathcal{R}_{\mathrm{speed}}q^6
$$
and
\begin{align*}
DL^T( [12;1],[\emptyset;12]) &= D( [12;1], [2;1]) L( [\emptyset;12], [2;1]) + D([12;1],[12,\emptyset])L([\emptyset;12],[12;\emptyset]) \\
&=q^4 \mathcal{R}_{\mathrm{TWIST}}(q) + q^4 \cdot q^2 \mathcal{R}_{\mathrm{SWAP}}(q),
\end{align*}
so duality requires
$$
q \mathcal{R}_{\mathrm{speed}} = \mathcal{R}_{\mathrm{TWIST}}(q) +q ^2 \mathcal{R}_{\mathrm{SWAP}} ,
$$
which holds for the given jump rates. 
\end{example}


\begin{example}
Likewise
$$
LD( [12;\emptyset], [2;1]) = (1+q^{-2\delta})\mathcal{R}_{\mathrm{TWIST}}(q^{-1}) \cdot \frac{1}{1+q^{2\delta}}q^6 
$$
and
$$
DL^T( [12;\emptyset], [2;1]) = q^4 \cdot q^{-2\delta} \mathcal{R}_{\mathrm{STICK}}(q),
$$
which requires
$$
q^2 \mathcal{R}_{\mathrm{TWIST}}(q^{-1})  =  \mathcal{R}_{\mathrm{STICK}}(q)
$$
\end{example}

\begin{remark}
The last three examples show that the five variables $\mathcal{R}_{\mathrm{speed}}, \mathcal{R}_{\mathrm{SWAP}}, \mathcal{R}_{\mathrm{STICK}}, \mathcal{R}_{\mathrm{TWIST}},\delta$ must satisfy at least three equations. Thus the set of solutions has at most two parameters. A natural conjecture is that for all $(n,\delta)$ the type $D$ ASEP with parameters $(n,\delta)$ is self--dual with respect to the duality function in Theorem \ref{dual}. 
\end{remark}

\begin{theorem}\label{Thm2}
The type D ASEP with parameters $(2,0)$ or $(3,0)$ is self--dual with respect to the function
$$
D(\eta,\xi) = 1_{\{ A_2(\xi)\subseteq A_2(\eta), A_1(\xi) \subseteq A_1(\eta)\}} \prod_{x \in A_1(\xi)} q^{2x - 2\cev{N}_x^1(\eta)} \prod_{x \in A_2(\xi)} q^{2x - 2\cev{N}_x^2(\eta)}.
$$
\end{theorem}

\begin{example}
Letting $D$ be the function in Theorem \ref{Thm2}, then
$$
LD([12;2],[2;1]) = q^{-1}\mathcal{R}_{\mathrm{speed}} \cdot q^6
$$
and
\begin{align*}
DL^T([12;2],[2;1]) &= D([12;2],[12;\emptyset][)L([2;1],[12;\emptyset]) + D([12;2],[1;2])L([1;2],[12,\emptyset])\\
&= q^4 \cdot \mathcal{R}_{\mathrm{SWAP}} + q^4 q^{-2\delta} \mathcal{R}_{\mathrm{STICK}}(q),
\end{align*}
which requires
$$
q \mathcal{R}_{\mathrm{speed}} = \mathcal{R}{_\mathrm{SWAP}} + q^{-2\delta} \mathcal{R}_{\mathrm{STICK}}(q),
$$
This only holds for $\delta=0$, demonstrating that the Theorem is false for $(n,\delta)$ if $\delta\neq 0$.
\end{example}

\section{Proofs}\label{Second}

\subsection{Proof of Theorem \ref{Center}}
The starting point of the the proof is the following lemma, which was stated in \cite{KuanJPhysA} and was based on \cite{jantzen}.

\begin{lemma}\label{main}
For each weight $\mu$ of the fundamental representation of a Lie algebra $\mathfrak{g}$, let $v_{\mu}$ be a vector in the weight space. Suppose $q$ is not a root of unity, and $2\mu$ is always in the root lattice of $\mathfrak{g}$. Let $e_{\mu\lambda}$ and $f_{\lambda\mu}$ be products of $E_i$s and $F_i$s in $U_q(\mathfrak{g})$, respectively, such that $e_{\mu\lambda}$ sends $v_{\lambda}$ to $v_{\mu}$ and $f_{\lambda\mu}$ sends $v_{\mu}$ to $v_{\lambda}$. If $e^*_{\mu\lambda}$ and $f^*_{\mu\lambda}$ are the corresponding dual elements under the above pairing, and $\rho$ is half the sum of the positive roots of $\mathfrak{g}$, then
\[
    \sum_{\mu} q^{(-2\rho, \mu)} q^{H_{-2\mu}} + \sum_{\mu > \lambda} q^{(\mu - \lambda, \mu)}q^{(-2\rho, \mu)} e^*_{\mu\lambda} q^{H_{-\mu-\lambda}} f^*_{\lambda \mu}
\]
is a central element of the quantum group $U_q(\mathfrak{g})$.
\end{lemma}

\subsubsection{Computing the individual terms}

This section describes each of the components of \cref{main} and efforts towards computing them explicitly. All results from here concern only $\mathfrak{so}_{2n}$.

First of all, we will take $q$ to be a real number different from $1$, so $q$ is not a root of unity. In addition, $2\mu = \pm 2 L_i$ is always in the root lattice, because $L_i \pm L_j$ are always roots. Thus, the conditions of the lemma are satisfied.

Our next step is to establish an ordering for the weight spaces to discern when $\mu > \lambda$ holds. Since the $E_i$ and $F_i$ operators serve as raising and lowering operators, and the sum of the positive roots is 
\[
    \sum_{i \ne j} (L_i + L_j) + \sum_{i<j} (L_i - L_j) = (2n-2) L_1 + (2n-4) L_2 + \cdots + L_{n-1},
\]
a natural ordering of the weights is 
\[
    L_1 > \cdots > L_{n-1} > L_n = -L_n > -L_{n-1} > \cdots > -L_1.
\]

\begin{lemma}\label{dbldyn}
The weight $\mu$ can be reached from a weight $\lambda$ by a product of $E_i$s, and the weight $\lambda$ can be reached from $\mu$ by a product of $F_i$s, if and only if $\mu > \lambda$ under the above ordering.
\end{lemma}
\begin{proof}
Define $v_1, \cdots, v_{2n}$ to be the vectors in the weight spaces $L_1, \cdots, L_n, -L_n, \cdots, -L_1$ with only a single nonzero entry of $1$. (Specifically, $v_i$ has a $1$ in the $i$th spot for $1 \le i \le n$ and a $1$ in the $(3n + 1 - i)$th spot for $n+1 \le i \le 2n$.) The action of the $E_i$s and $F_i$s can be described explicitly based on their matrix counterparts in the fundamental representation, defined earlier: for all $1 \le i \le n-1$, 
\[
    E_i v_j = \begin{cases} v_i & j = i+1 \\ -v_{2n-i} & j = 2n+1-i \\ 0 & \text{otherwise,} \end{cases}
\]
meaning that the first $(n-1)$ $E$ operators ``move us up one weight space,'' and
\[
    E_n v_j = \begin{cases} v_{n-1} & j = n+1 \\ -v_n & j = n + 2 \\ 0 & \text{otherwise.} \end{cases}
\]
Below is a schematic diagram for the weight spaces, which resembles the Dynkin diagram $D_n$ mirrored over itself:
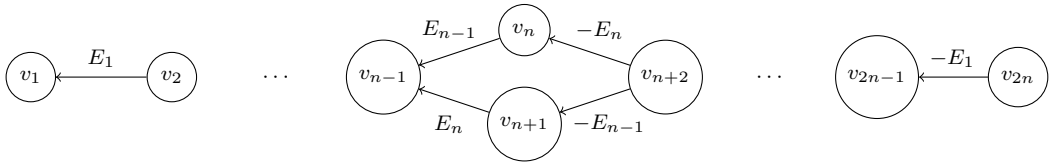
\begin{center}
\begin{tikzpicture}[scale=1.25, font = \footnotesize]
\node[circle, draw] (d) at (0, 0.5) {$v_n$};
\node[circle, draw] (e) at (0, -0.5) {$v_{n+1}$};
\node[circle, draw] (f) at (1.5, 0) {$v_{n+2}$};
\node[circle, draw] (c) at (-1.5, 0) {$v_{n-1}$};
\node[circle, draw] (g) at (3.75, 0) {$v_{2n-1}$};
\node[circle, draw] (b) at (-3.75, 0) {$v_{2}$};
\node[circle, draw] (h) at (5.25, 0) {$v_{2n}$};
\node[circle, draw] (a) at (-5.25, 0) {$v_{1}$};
\draw[<-] (a) -- (b);
\draw[<-] (g) -- (h);
\draw[<-] (c) -- (d);
\draw[<-] (c) -- (e);
\draw[<-] (d) -- (f);
\draw[<-] (e) -- (f);
\node at (-2.625, 0) {$\cdots$};
\node at (2.625, 0) {$\cdots$};
\node at (-4.5, 0.2) {$E_1$};
\node at (4.55, 0.2) {$-E_1$};
\node at (-0.8, 0.5) {$E_{n-1}$};
\node at (-0.8, -0.5) {$E_{n}$};
\node at (0.8, 0.5) {$-E_{n}$};
\node at (0.9, -0.5) {$-E_{n-1}$};
\end{tikzpicture}
\end{center}

The result can now be directly verified for the $E_i$s by inspection of the diagram (the ordering goes left-to-right, and the $E_i$s always move us left in the diagram). The proof for $F_i$s follows similarly.
\end{proof}

Since $E_{n-1}$ and $E_n$ commute, and so do $F_{n-1}$ and $F_n$, taking the ``top path'' or the ``bottom path'' in the diagram above in fact corresponds to applying the same element of $U(\mathfrak{g})$. This means that there is in fact a \textbf{unique} product of $E$s or $F$s (up to scalars) that takes any $v_i$ to any $v_j$ (whenever $j$ can be reached from $i$), and this means that we have $\binom{2n}{2} - 1$ total pairs $(\mu, \lambda)$ in the second sum of \cref{main}. 

Now that we know how to find the elements $e_{\mu\lambda}$ and $f_{\lambda \mu}$, we can begin to calculate the quantities in the actual sum. Computing $q^{(\mu - \lambda, \mu)}$ and $q^{(-2\rho, \mu)}$ is relatively simple, because $\mu$ always consists of a single term of the form $\pm L_i$:
\[
    q^{(\mu - \lambda, \mu)} = q^{(\mu, \mu) - (\lambda, \mu)} = \begin{cases} q^2 & \lambda = -\mu \\ 1 & \lambda = \mu \\ q & \text{otherwise,} \end{cases} \quad q^{(-2\rho, \mu)} = \begin{cases} q^{2n - 2i} & \mu = L_i \\ q^{2i - 2n} & \mu = -L_i. \end{cases}
\]
Computing the others requires more casework. To find $q^{H_{-2\mu}}$ and $q^{H_{-\mu-\lambda}}$, we need to write the exponents as linear integer combinations of the $H_i$s. We present the general argument here and show examples later in the paper.

\begin{itemize}
\item To compute $q^{H_{-2\mu}}$ when $\mu = \pm L_i$, we first find that (with some abuse of notation)
\[
    \boxed{H_{-2L_n}} = H_{-(L_{n-1} - L_n) + (L_{n-1} + L_n)} = \boxed{H_{n-1} - H_{n}},
\]
and then use a telescoping sum: since
\[
    H_{-2L_i} - H_{-2L_n} = \sum_{j = i}^{n-1} (-2L_j + 2L_{j+1}) = -2\sum_{j=i}^{n-1} H_j,
\]
we can rearrange and substitute to
\[
    \boxed{H_{-2L_i} = H_{n-1} - H_n -2 \sum_{j=i}^{n-1} H_j}
\]
for any $1 \le i \le n-1$. This gives us $\mu = +L_i$, and we simply negate the expressions for the case $\mu = -L_i$. (For example, $n = 3, \mu = L_2$ yields $H_{-2\mu} = -H_2 - H_3$, so $\mu = -L_2$ yields $H_{-2\mu} = H_2 + H_3$.)
\item A similar telescoping sum works when $\mu > \lambda$ and we want to find $q^{H_{-\mu-\lambda}}$. (Such exponents always look like $H_{\pm L_i \pm L_j}$ for $i \ne j$.) First, note that when $i < j$,
\[
    \boxed{H_{L_i - L_j}} = \sum_{k = i}^{j-1} H_{L_k - L_{k+1}} = \boxed{\sum_{k=i}^{j-1} H_k},
\]
and the $i > j$ case follows by negating both expressions. This allows us to find $H_{\pm L_i \pm L_j}$ by subtracting or adding $H_{2L_i}$ and $H_{2L_j}$, which have been computed above.
\end{itemize}

We now turn our attention to computing $e_{\mu\lambda}, f_{\lambda\mu}$, and their dual elements. Since all $e_{\mu\lambda}$ and $f_{\lambda\mu}$s are products of $E_i$s and $F_i$s, respectively, the dual elements $e^*_{\mu\lambda}$ and $f^*_{\lambda\mu}$ will be products of $F_i$s and $E_i$s. The next result serves to characterize these elements, and it also explains more explicitly how the above bilinear pairing is calculated.

\begin{lemma}
The value of $\langle F_{x_1}F_{x_2} \cdots F_{x_m}, E_{y_1} E_{y_2} \cdots E_{y_n} \rangle$ is nonzero if and only if $(x_1, \cdots, x_m)$ and $(y_1, \cdots, y_n)$ are permutations of each other (in particular, $n = m$), in which case it evaluates to $(q - q^{-1})^{-n}$ times an element of the ring $\mathbb{Z}[q, q^{-1}]$.
\end{lemma}
\begin{proof}
We proceed by induction on $n$. The base case $n = 1$ can be verified using the rule $\langle yy', x \rangle = \langle y \otimes y', \Delta(x)\rangle$ and the fact that $1$ and $q^{H_i}$ both pair to $0$ with any $F_i$ or product of $F_i$s. Then $\langle F_i, E_j \rangle$ is always equal to $-\delta_{ij}(q-q^{-1})^{-1}$, which proves the remainder of the claim.

For the inductive step, using the coproduct relation, the above pairing evaluates to
\[
    \left\langle \prod_{i=1}^n \left(1 \otimes F_{x_i} + \boxed{F_{x_i} \otimes q^{-H_{x_i}}} \right), E_{y_2} \cdots E_{y_n} \otimes E_{y_1} \right\rangle.
\]
By the inductive hypothesis, the only way this pairing is nonzero is if $(n-1)$ of the terms on the left are of the boxed type (so that there are exactly $(n-1)$ $F$s to pair with the $(n-1)$ $E$s in $E_{y_2} \cdots E_{y_n}$). In addition, to yield a nonzero result, the remaining term (of the form $(1 \otimes F_{x_i})$) must have the same index as $E_{y_1}$. Thus, we've proved the first claim, and this expression evaluates to 
\[
    \sum_{i:\, x_i = y_1} \langle F_{x_1} \cdots \hat{F_{x_i}} \cdots F_{x_n}, E_{y_2} \cdots E_{y_n} \rangle \langle q^{-H_{x_1}} \cdots F_{x_i} \cdots q^{-H_{x_n}}, E_{y_1} \rangle,
\]
where the hat means that $F_{x_i}$ is omitted. By the inductive hypothesis again, the first pairing here is always $(-(q-q^{-1})^{-1})^{n-1}$ times an element of $\mathbb{Z}[q, q^{-1}]$, and the second pairing is (moving the $F$ to the front)
\[
    \langle q^{-H_{x_1}} \cdots F_{x_i} \cdots q^{-H_{x_n}}, E_{y_1} \rangle = q^{(\alpha_{x_i}, \alpha_{x_1} + \cdots + \alpha_{x_{i-1}})} \langle F_{x_i} q^{-H_{x_1}} \cdots \hat{q^{-H_{x_i}}} \cdots q^{-H_{x_n}}, E_{y_1} \rangle
\]
by the commutativity relations of the quantum group. Now using the coproduct relation again on this final pairing and substituting shows that it is always $-(q-q^{-1})^{-1}$, so we have the right power of $(q - q^{-1})$, as desired.
\end{proof}

Now, computing the dual elements can be done as follows. For a given set of indices $(x_1, \cdots, x_n)$, all elements of the form $e' = E_{\sigma(x_1)} \cdots E_{\sigma(x_n)}$ (for a permutation of the indices $\sigma$) have nonzero pairing only with elements of the form $f' = F_{\tau(x_1)} \cdots F_{\tau(x_n)}$ (for a permutation $\tau$). However, some of these elements may be identical or linearly dependent (due to the relations of the quantum group). Thus, pick $\{e_1, \cdots, e_m\}$ and $\{f_1, \cdots, f_m\}$ to be (linear) bases of the spaces of possible $e'$s and $f'$s, and let $M$ be the matrix such that $M_{ij} = \langle e_i, f_j \rangle$. Results from Chapter 6 of \cite{jantzen} show that this pairing is nondegenerate, so $M$ must be an invertible matrix. The rows of $M^{-1}$ then tell us the correct $\mathbb{Z}[q, q^{-1}]$-combinations to take for the dual element.

\begin{example}\label{dualex}
Take $n = 4$. If we wish to find the dual element for $E_2E_3$, first note that $E_2E_3$ and $E_3E_2$ both have nonzero pairing with $F_2F_3$ and $F_3F_2$, and none of these elements have a nonzero pairing with anything else. Taking the bases to be $\{E_2E_3, E_3E_2\}$, $\{F_2F_3, F_3F_2\}$, we evaluate the pairings to find
\[
    M = (q-q^{-1})^{-2} \begin{bmatrix} 1 & 1/q \\ 1/q & 1 \end{bmatrix} \implies M^{-1} = (q - q^{-1})^2 \frac{1}{q-q^{-1}} \begin{bmatrix} q & -1 \\ -1 & q \end{bmatrix}.
\]
Thus, reading off the first row, $E_2E_3$ has dual element $(q - q^{-1}) (qF_2F_3 - F_3F_2)$.
\end{example}

In the case of $\mathfrak{so}_6$ and $\mathfrak{so}_8$, this procedure is done with the aid of a computer. See the appendix for the relevant Python code.

\subsection{Proofs of Proposition \ref{rev}, Theorem \ref{dual} and Theorem \ref{Thm2}}
The starting point is the following general method introduced in \cite{CGRS}.  Suppose that we are given a self--adjoint Hamiltonian $H$, a ground state $g$ (i.e. a vector such that $Hg=0$), and a symmetry $S$ of $H$ (i.e. $HS=SH$). Let $G$ be the diagonal matrix with entries given by the entries of $g$. Then the rows of $L=G^{-1}HG$ sum to $0$, so $L$ is the generator of a Markov process if the off--diagonal entries are non--negative. 
Then $G^2$ defines reversible measures for the Markov process, and the Markov process is self--dual with respect to the function $D=G^{-1}SG^{-1}$. 

The Hamiltonian will be constructed from the central elements of Theorem \ref{Center}. The proof will be split up into the $\mathcal{U}_q(\mathfrak{so}_6)$ case and the $\mathcal{U}_q(\mathfrak{so}_8)$ case.

\subsubsection{The $\mathcal{U}_q(\mathfrak{so}_6)$ case}
Consider the central element $C$ in part (a) of Theorem \ref{Center}. Let $H$ denote the action of $\Delta(C)$ on $\mathbb{C}^6 \otimes \mathbb{C}^6$, so that $H$ is a $36\times 36$ matrix. With the aid of a computer, one sees that there is a constant $\mathrm{const}$ such that $H-\mathrm{const}\cdot \mathrm{Id}$ decomposes as a direct sum of $1\times 1$ blocks with entry zero, $2\times 2$ blocks with entries
$$
\left(
\begin{array}{cc}
-q^4-q^{-6} & q^5 + q^{-5} \\
 q^5 + q^{-5} & -q^6 - q^{-4} 
\end{array}
\right)
$$
and a single $6 \times 6$ block
$$
\left(
\begin{array}{cccccc}
 -\frac{\left(q^2+1\right)^2 \left(q^6-q^4+1\right)}{q^6} & -q^3-\frac{2}{q^3}+q &
   \frac{-q^6+q^4-2}{q^2} & q^4+\frac{1}{q^4}-2 & \frac{-q^6+q^4-2}{q} &
   \frac{-q^6+q^4-2}{q^2} \\
 -q^3-\frac{2}{q^3}+q & \frac{-2 q^{10}+q^8-2 q^4-1}{q^6} & \frac{-q^6+q^4-2}{q} &
   \frac{1}{q^3}-\frac{1}{q^5}-2 q & q^4+\frac{1}{q^4}-2 & \frac{-q^6+q^4-2}{q} \\
 \frac{-q^6+q^4-2}{q^2} & \frac{-q^6+q^4-2}{q} & -\frac{\left(q^6+1\right)^2}{q^6} &
   \frac{-2 q^6+q^2-1}{q^4} & \frac{1}{q^3}-\frac{1}{q^5}-2 q & q^4+\frac{1}{q^4}-2
   \\
 q^4+\frac{1}{q^4}-2 & \frac{1}{q^3}-\frac{1}{q^5}-2 q & \frac{-2 q^6+q^2-1}{q^4} &
   -\frac{\left(q^2+1\right)^2 \left(q^6-q^2+1\right)}{q^4} & \frac{-2
   q^6+q^2-1}{q^3} & \frac{-2 q^6+q^2-1}{q^4} \\
 \frac{-q^6+q^4-2}{q} & q^4+\frac{1}{q^4}-2 & \frac{1}{q^3}-\frac{1}{q^5}-2 q &
   \frac{-2 q^6+q^2-1}{q^3} & -\frac{q^{10}+2 q^6-q^2+2}{q^4} &
   \frac{1}{q^3}-\frac{1}{q^5}-2 q \\
 \frac{-q^6+q^4-2}{q^2} & \frac{-q^6+q^4-2}{q} & q^4+\frac{1}{q^4}-2 & \frac{-2
   q^6+q^2-1}{q^4} & \frac{1}{q^3}-\frac{1}{q^5}-2 q &
   -\frac{\left(q^6+1\right)^2}{q^6} \\
\end{array}
\right).
$$ 
The $6\times 6$ matrix is written with respect to the ordered basis 
$$
(v_1 \otimes v_6, v_2 \otimes v_5, v_3 \otimes v_4, v_6 \otimes v_1, v_5 \otimes v_2, v_4 \otimes v_3).
$$

The next step is to find an eigenvector $g$ of $H$ such that $Hg=0$. It is clear that the entries corresponding to the $1\times 1$ blocks may be arbitrary, and the entries corresponding to the $2\times 2$ blocks have entries proportional to 
$$
\left( 
\begin{array}{c} 
q \\
1
\end{array}
\right)
$$
The $6\times 6$ matrix has rank $4$ and has two linearly independent eigenvectors 
$$
g_1:=
\left(
\begin{array}{c}
0 \\
q^2\\
-q\\
0\\
1 \\
-q
\end{array}
\right), \quad
g_2:=
\left(
\begin{array}{c}
q^2 \\
-q\\
0\\
1\\
-q \\
0
\end{array}
\right)
$$
with eigenvalue $0$. There are thus two linearly independent choices for the ground state vector $g$, which correspond to the two type D ASEPs with parameters $(3,0)$ and $(3,1)$, respectively. Note that $g_1$ and $g_2$ have the same nonzero entries in the same order -- this will eventually imply that the two processes have the same reversible measures. These two cases will be treated separately below.

\paragraph{The $g_2$ case}

We can now obtain the Markov process by conjugating the Hamiltonian. The most nontrivial interaction is in the $6\times 6$ block.  Let $G_i$ denote the $6\times 6$ diagonal matrix with entries given by the entries of $g_i$. Then conjugating the $6\times 6$ matrix by $G_2 $ yields
$$
\left(
\begin{array}{cccccc}
 -\frac{\left(q^2+1\right)^2 \left(q^6-q^4+1\right)}{q^6} & \frac{q^6-q^4+2}{q^4} & 0
   & \frac{\left(q^4-1\right)^2}{q^6} & \frac{q^6-q^4+2}{q^2} & 0 \\
 \frac{q^6-q^4+2}{q^2} & \frac{-2 q^{10}+q^8-2 q^4-1}{q^6} & 0 & \frac{2
   q^6-q^2+1}{q^6} & \frac{\left(q^4-1\right)^2}{q^4} & 0 \\
 \infty & \infty & -\frac{\left(q^6+1\right)^2}{q^6}
   & \infty & \infty & q^4+\frac{1}{q^4}-2 \\
 \frac{\left(q^4-1\right)^2}{q^2} & \frac{2 q^6-q^2+1}{q^4} & 0 &
   -\frac{\left(q^2+1\right)^2 \left(q^6-q^2+1\right)}{q^4} & \frac{2 q^6-q^2+1}{q^2}
   & 0 \\
 q^6-q^4+2 & q^4+\frac{1}{q^4}-2 & 0 & \frac{2 q^6-q^2+1}{q^4} & -\frac{q^{10}+2
   q^6-q^2+2}{q^4} & 0 \\
 \infty & \infty & q^4+\frac{1}{q^4}-2 &
   \infty & \infty &
   -\frac{\left(q^6+1\right)^2}{q^6} \\
\end{array}
\right)
$$
One can check that $2q^6-q^2+1$ and $q^6-q^4+2$ are positive for all real values of $q$. Thus, this matrix defines the generator for a Markov process with four states (assuming that the initial condition has probability zero of being in states $3$ or $6$). This matrix is obtained by simply removing the third and sixth rows and columns: 
$$
\left(
\begin{array}{cccc}
 -\frac{\left(q^2+1\right)^2 \left(q^6-q^4+1\right)}{q^6} & q^2-1+2q^{-4}
   & \frac{\left(q^4-1\right)^2}{q^6} & q^4-q^2+2q^{-2}  \\
 q^4-q^2+2q^{-2} & \frac{-2 q^{10}+q^8-2 q^4-1}{q^6}  & 2-q^{-4}+q^{-6} & \frac{\left(q^4-1\right)^2}{q^4}  \\
 \frac{\left(q^4-1\right)^2}{q^2} & 2 q^2-q^{-2}+q^{-4}  &
   -\frac{\left(q^2+1\right)^2 \left(q^6-q^2+1\right)}{q^4} & 2 q^4-1+q^{-2}
    \\
 q^6-q^4+2 & q^4+q^{-4}-2  & 2 q^2-q^{-2}+q^{-4} & -\frac{q^{10}+2
   q^6-q^2+2}{q^4}  
\end{array}
\right)
$$
This is now written with respect to the basis
$$
(v_1 \otimes v_6, v_2 \otimes v_5, v_6 \otimes v_1, v_5 \otimes v_2).
$$
We associate to each vector $v_1,v_2,v_5,v_6$ a particle configuration as follows:

\begin{center}
\begin{tikzpicture}
    \node at (-5,0){$v_1=$};
    \node[draw,fill=white] at (-4,-0.25){$1$};
    \node[draw,fill=white] at (-4,0.25){$2$};
    
    \node at (-2.5,0){$v_2=$};
    \node[draw,fill=white] at (-1.5,-0.25){$2$};
    \node[circle,draw,fill=white] at (-1.5,0.25){};
    
     \node at (0,0){$v_5=$};
    \node[draw,fill=white] at (1,-0.25){$1$};
    \node[circle,draw,fill=white] at (1,0.25){};
    
    \node at (2.5,0){$v_6=$};
    \node[circle,draw,fill=white] at (3.5,0.25){};
    \node[circle,draw,fill=white] at (3.5,-0.25){};
\end{tikzpicture}
\end{center}

The jump rates can be read off by noting that
$$ 
\left(
\begin{array}{cccc}
  *& \text{TWIST-RIGHT} & 
   \text{SWAP} & \text{TWIST-RIGHT}\\
 \text{STICK-LEFT} & *  & 
\text{STICK-RIGHT}&  \text{SWAP}\\

  \text{SWAP} & \text{TWIST-LEFT} & 
  *& \text{TWIST-LEFT} \\
 \text{STICK-LEFT} &  \text{SWAP} & \text{STICK-RIGHT}&
* \\
\end{array}
\right)
$$
Thus, one sees that the resulting Markov process is the type D ASEP with parameters $(2,1)$. 

The next step is to extend the ground state vector from $2$ sites to $N$ sites. This will be constructed algebraically. For example, for $2$ sites the vector $g_2$ can be constructed from 
\begin{align*}
-q E_1E_2E_3E_1(v_6 \otimes v_6) &= -qv_1 \otimes v_6 + v_2 \otimes v_5 - q^{-1}v_6 \otimes v_1 + v_5 \otimes v_2.
\end{align*}
yields the eigenvector $g_2$. The most natural extension of the ground state vector to $N$ sites is to define $v_6^{\otimes N}$ to be the vacuum vector (because $v_6$ is the lowest weight vector), and apply creation operators $E_j$'s to the vacuum vector. See e.g. \cite{KuanJPhysA} where this is done for the Lie algebra $\mathfrak{sp}_4$.

However, this will not extend to $N$ sites; this is because the $E_j$'s potentially create the vectors $v_3$ and $v_4$, which need to be avoided. Instead, note that
$$
-qF_1E_1F_3E_2(v_3 \otimes v_3) = -qv_1 \otimes v_6 + v_2 \otimes v_5- q^{-1}v_6 \otimes v_1    +v_5 \otimes v_2
$$
and define the ``vacuum'' vector $\Omega_N = v_3^{\otimes N}$. There is now an additional difficulty, because both creation and annihilation operators are needed.

\begin{proposition}
For any $M,K,N\geq 0$, 
$$
F_1^KE_1^MF_3^{N-M}E_2^M\vert \Omega_N \rangle = \sum_{\eta} G(\eta)\vert \eta\rangle,
$$
where the sum is over all particle configurations $\eta$ on $N$ sites with $M$ second class particles and $N-K$ first class particles, and 
$$
\vert G(\eta) \vert= Z_{N,M,K}^{-1} \prod_{x_1 \in A_1(\eta)} q^{-x_1} \prod_{x_2\in A_2(\eta)} q^{-x_2}
$$
for some normalization constant $Z_{N,M,K}$.
\end{proposition}
\begin{proof}
When expanding the coproduct on each $E_i,F_i$, each term in the tensor product is  $1,E_i,F_i$ or $K_i^{\pm 1}$. The $E_i$ and $F_i$ act on the vectors $\{v_1,\ldots,v_6\}$ with nonzero coefficients $\pm 1$. By taking the absolute value of $G(\eta)$, the signs can be ignored. Additionally, the fact that the summation is over the specified $\eta$ follows from the identification of $\{v_1,v_2,v_4,v_5\}$ with particle configurations.

To show that $\vert G(\eta)\vert$ has the given expression, compare $\vert G(\eta)\vert$ with $\vert G(\eta') \vert$, where $\eta$ is obtained from $\eta$ by moving one particle to the right. If the moved particle is a second class particle, then $\vert G(\eta') \vert$ is altered in the following way:
\begin{itemize}
\item
The $E_2$ term produces an extra $q^{-1}$, because now $K_2$ acts on an extra $v_3$ vector.
\item
The $F_3$ term produces an extra $q^{-1}$, because now $K_3^{-1}$ acts on an extra $v_2$ vector.
\item
The $E_1$ term produces an extra $q$, because now $K_1$ acts on an extra $v_5$ vector.
\item
The $F_1$ term contribution is unchanged, because $F_1$ only acts on first--class particles.
\end{itemize}
Multiplying the four contributions above shows that $\vert G(\eta')\vert$ is multiplied by $q^{-1}$ when a second class particle is moved. This is the same as increasing one of the $x_2\in A_2(\eta)$ by $1$, explaining the appearance of  $\prod_{x_2\in A_2(\eta)} q^{-x_2}$. A similar consideration yields $\prod_{x_1 \in A_1(\eta)} q^{-x_1}$.
\end{proof}



This proposition implies Proposition \ref{rev} for the type $D$ ASEP with parameters $(2,1)$. 

\begin{remark}
The corollary can be summarized as the statement that $q^2$--exchangeable measures are reversible. This is identical to the case for multi--species ASEP \cite{BS3}, multi--species ASEP$(q,j)$ \cite{KIMRN} and multi--species $q$--TAZRP \cite{KuanAHP}.

\end{remark}

Next we prove Theorem \ref{dual} for the type $D$ ASEP with parameters $(2,1)$.

Recall the $q$--exponential
$$ \exp _{q}(x) :=\sum_{n \geq 0} \frac{x^{n}}{\{n\}_{q} !}$$
where
$$
\{n\}_{q} :=\frac{1-q^{n}}{1-q} $$
Let $$S(\eta,\xi) = \langle \xi \vert \exp_{q^2}(\Delta^{(N-1)}(F_1)) \vert \eta \rangle.$$
Then Proposition 5.1 of \cite{CGRS} implies that
$$
S(\eta,\xi) = \langle \xi \vert F_1^{(1)} \cdots F^{(L)}_1\vert \eta \rangle
$$
for 
$$
F_1^{(j)} = 1^{\otimes j-1} \otimes F_1 \otimes (K_1^{-1})^{\otimes N-j}.
$$
Furthermore, the general framework of \cite{CGRS} then says that
$$
D(\eta,\xi) = G^{-1}(\eta)S(\eta,\xi)G^{-1}(\xi)
$$
is a duality function. First, note that $S(\eta,\xi)$ is nonzero if and only if the indicator function is nonzero. Second, note a duality function can be multiplied by quantities that are constant with respect to the dynamics, so thus the $Z_{N,M,K}$ terms in $G$ can be ignored. To see the contribution from the $G^{-1}$ terms, set $\delta = \eta-\xi$ to be the particle configuration consisting of particles in $\eta$ but not in $\xi$. Then
$$
G^{-1}(\eta)G^{-1}(\xi) = G^{-1}(\delta)G^{-2}(\xi). 
$$
The term $G^{-2}(\xi)$ contributes (up to constants)
$$
\prod_{x \in A_1(\xi)} q^{2x} \prod_{x \in A_2(\xi)} q^{2x}.
$$
Third, note that $F_1$ is applied at every particle in $\delta$. In other words, $F_1$ is applied at every lattice site $y\in A_1(\delta)$. The $K_1^{-1}$ makes a contribution for every $x$ to the right of $y$, with value
$$
\begin{cases}
q^{-1}, \text{ for } x \in A_1(\xi)\cap A_2(\xi)\\
q, \text{ for } x\in A_2(\xi)-A_1(\xi) \\
q, \text{ for } x\notin A_1(\xi)\cup A_2(\xi) \\
q^{-1}, \text{ for } x\in A_1(\xi)-A_2(\xi).
\end{cases}
$$
corresponding to the four vectors $v_1,v_2,v_4,v_5$ respectively. The $G^{-1}(\delta)$ then contributes $q^{-1}$ for every $x$ to the right of $y$. Multiplying $q^{-1}$ to each of the four terms above, what remains is just $q^{-2}$ for every $x\in A_1(\xi)$ and $y\in A_1(\delta)$. This results in 
$$
\prod_{x\in A_1(\xi)} q^{-2\cev{N}_x^1(\delta)} = \prod_{x\in A_1(\xi)} q^{-2\cev{N}_x^1(\eta)+2\cev{N}_x^1(\xi)}  .
$$
Since $ \prod_{x\in A_1(\xi)} q^{2\cev{N}_x^1(\xi)} $ is constant with respect to the dynamics, what remains is the $-2\cev{N}_x^1(\eta)$ contribution. Combining all contributions gives the function $D(\eta,\xi)$, as needed.

\paragraph{The $g_1$ case}
We can obtain the entries of the vector $g_1$ via
$$
 -q^2(-q+q^{-1})^{-1}E_2E_3E_1^2(v_6 \otimes v_6) =q^2v_2 \otimes v_5 - q^{}v_3 \otimes v_4 - q^{}v_4 \otimes v_3 + v_5\otimes v_2
$$

Analogously to before, conjugating the $6\times 6$ matrix by $G_1$ and removing the first and fourth rows and columns yields
$$ L_1=
\left(
\begin{array}{cccc}
  \frac{-2 q^{10}+q^8-2 q^4-1}{q^6} & q^4-q^2+2q^{-2}& 
   \frac{\left(q^4-1\right)^2}{q^6} & q^4-q^2+2q^{-2}\\
  q^6-q^4+2 & -\frac{\left(q^6+1\right)^2}{q^6} & 
  -q^{-4}+q^{-6}+2 & q^4+q^{-4}-2 \\

  \frac{\left(q^4-1\right)^2}{q^2} & 2 q^2-q^{-2}+q^{-4} & 
   -\frac{q^{10}+2 q^6-q^2+2}{q^4} & 2 q^2-q^{-2}+q^{-4} \\
  q^6-q^4+2 & q^4+q^{-4}-2 & -q^{-4}+q^{-6}+2 &
   -\frac{\left(q^6+1\right)^2}{q^6} \\
\end{array}
\right)
$$
with respect to the basis 
$$
(v_2 \otimes v_5, v_3 \otimes v_4, v_5 \otimes v_2, v_4 \otimes v_3).
$$
Associate to each vector a particle configuration by 
\begin{center}
\begin{tikzpicture}
    \node at (-5,0){$v_2=$};
    \node[draw,fill=white] at (-4,-0.25){$1$};
    \node[draw,fill=white] at (-4,0.25){$2$};
    
    \node at (-2.5,0){$v_3=$};
    \node[draw,fill=white] at (-1.5,-0.25){$2$};
    \node[circle,draw,fill=white] at (-1.5,0.25){};
    
     \node at (0,0){$v_4=$};
    \node[draw,fill=white] at (1,-0.25){$1$};
    \node[circle,draw,fill=white] at (1,0.25){};
    
    \node at (2.5,0){$v_5=$};
    \node[circle,draw,fill=white] at (3.5,0.25){};
    \node[circle,draw,fill=white] at (3.5,-0.25){};
\end{tikzpicture}
\end{center}

In this case, we let $v_5$ be the vacuum state. So we start from the vacuum vector $\Omega_N=v_5^{\otimes N}$. Then $E_2$ is the operator that produce a particle of class one, and $E_3$ is the  operator that produce a particle of class two, and $F_2$,$F_3$ are annihilation operators for particles of class one and two respectively. 

\begin{proposition}
For any $M_1,M_2,N>0$,
\begin{equation}
    E_2^{M_1}E_3^{M_2}\vert \Omega_N \rangle=  \sum_{\eta} G(\eta)\vert \eta\rangle,
\end{equation}
where the sum is over all particle configurations $\eta$ on $N$ sites with $M_1$ first class particles and $M_2$ second class particles, and 
$$
\vert G(\eta) \vert= Z_{N,M_1,M_2}^{-1} \prod_{x_1 \in A_1(\eta)} q^{-x_1} \prod_{x_2\in A_2(\eta)} q^{-x_2}
$$
for some normalization constant $Z_{N,M_1,M_2}$.
\end{proposition}
\begin{proof}
This is proved by direct computation.
\end{proof}
This proposition implies Proposition \ref{rev} for the type $D$ ASEP with parameters $(2,1)$. 

To find a duality function, we have several natural symmetries, for example, $exp_{q^2}(\Delta^{(N-1)}(E_2))$, $exp_{q^2}(\Delta^{(N-1)}(E_3))$, $exp_{q^{2}}(\Delta^{(N-1)}(F_2))$, $exp_{q^{2}}(\Delta^{(N-1)}(F_3))$.
Here we take 
$$S(\eta,\xi) = \langle \xi \vert \exp_{q^{2}}(\Delta^{(N-1)}(F_2))exp_{q^{2}}(\Delta^{(N-1)}(F_3)) \vert \eta \rangle.$$
Then, $S(\eta,\xi)$ is nonzero if and only if $\xi$ is obtained by deleting some particles of class one and two in $\eta$. Again, denote $\delta=\eta-\xi$. Thus, $F_2$ is applied at every lattice site $y\in A_1(\delta)$ and $F_3$ is applied at every lattice site $y\in A_2(\delta)$. The $K_2^{-1}$ makes a contribution for every $x$ to the right of $y\in A_1(\delta)$, with value
$$
\begin{cases}
q^{-1}, \text{ for } x \in A_1(\xi)\cap A_2(\xi)\\
q, \text{ for } x\in A_2(\xi)-A_1(\xi) \\
q, \text{ for } x\notin A_1(\xi)\cup A_2(\xi) \\
q^{-1}, \text{ for } x\in A_1(\xi)-A_2(\xi).
\end{cases}
$$
corresponding to the four vectors $v_2,v_3,v_5,v_4$ respectively.

 The $K_3^{-1}$ makes a contribution for every $x$ to the right of $y\in A_2(\delta)$, with value
$$
\begin{cases}
q^{-1}, \text{ for } x \in A_1(\xi)\cap A_2(\xi)\\
q^{-1}, \text{ for } x\in A_2(\xi)-A_1(\xi) \\
q, \text{ for } x\notin A_1(\xi)\cup A_2(\xi) \\
q, \text{ for } x\in A_1(\xi)-A_2(\xi).
\end{cases}
$$
corresponding to the four vectors $v_2,v_3,v_5,v_4$ respectively.

 The $G^{-1}(\delta)$ then contributes $q^{-1}$ for every $x$ to the right of $y\in A_1(\delta)$ and every $x$ to the right of $y\in A_2(\delta)$ respectively. Multiplying $q^{-1}$ to each of the eight terms above, what remains is just $q^{-2}$ for every $x\in A_1(\xi)$ with $y\in A_1(\delta)$ and $q^{-2}$ for every $x\in A_2(\xi)$ with $y\in A_2(\delta)$. This results in 
$$
\prod_{x\in A_1(\xi)} q^{-2\cev{N}_x^1(\eta)}\prod_{x\in A_2(\xi)} q^{-2\cev{N}_x^2(\eta)}.
$$

This proves Theorem \ref{Thm2} for the type $D$ ASEP with parameters $(2,0)$.

\subsubsection{The $\mathcal{U}_q(\mathfrak{so}_8)$ case}
As before, let $C$ be the central element from Theorem \ref{Center} and let $H$ be the action of $\Delta(C)$ on $\mathbb{C}^8 \otimes \mathbb{C}^8$. Then $H$ is a $64 \times 64$ matrix. With the aid of a computer, one sees that there is a constant $\mathrm{const}$ such that $H-\mathrm{const}\mathrm{Id}$ decomposes as a direct sum of $1\times 1$ blocks with entry zero, $2\times 2$ blocks with entries
$$
\left(
\begin{array}{cc}
-q^6-q^{-8} & q^7 + q^{-7} \\
 q^7 + q^{-7} & -q^8 - q^{-6} 
\end{array}
\right)
$$
and a single $8 \times 8$ block which has the form $U^T + D + U$, where 
$$
U=\left(
\begin{array}{cccccccc}
 0 & \frac{-q^8+q^6-2}{q^5} & \frac{-q^8+q^6-2}{q^4} & \frac{-q^8+q^6-2}{q^3} &
   q^6+\frac{1}{q^6}-2 & \frac{-q^8+q^6-2}{q} & \frac{-q^8+q^6-2}{q^2} &
   \frac{-q^8+q^6-2}{q^3} \\
 0 & 0 & \frac{-q^8+q^6-2}{q^3} & \frac{-q^8+q^6-2}{q^2} & \frac{-2 q^8+q^2-1}{q^7} &
   q^6+\frac{1}{q^6}-2 & \frac{-q^8+q^6-2}{q} & \frac{-q^8+q^6-2}{q^2} \\
 0 & 0 & 0 & \frac{-q^8+q^6-2}{q} & \frac{-2 q^8+q^2-1}{q^6} & \frac{-2
   q^8+q^2-1}{q^7} & q^6+\frac{1}{q^6}-2 & \frac{-q^8+q^6-2}{q} \\
 0 & 0 & 0 & 0 & \frac{-2 q^8+q^2-1}{q^5} & \frac{-2 q^8+q^2-1}{q^6} & \frac{-2
   q^8+q^2-1}{q^7} & q^6+\frac{1}{q^6}-2 \\
 0 & 0 & 0 & 0 & 0 & \frac{-2 q^8+q^2-1}{q^3} & \frac{-2 q^8+q^2-1}{q^4} & \frac{-2
   q^8+q^2-1}{q^5} \\
 0 & 0 & 0 & 0 & 0 & 0 & \frac{-2 q^8+q^2-1}{q^5} & \frac{-2 q^8+q^2-1}{q^6} \\
 0 & 0 & 0 & 0 & 0 & 0 & 0 & \frac{-2 q^8+q^2-1}{q^7} \\
 0 & 0 & 0 & 0 & 0 & 0 & 0 & 0 \\
\end{array}
\right)
$$
and $D$ is the diagonal matrix with entries
\begin{multline*}
\Big\{-q^6-q^2-\frac{2}{q^6}-\frac{1}{q^8}+1,\frac{-q^{14}-q^{12}+q^{10}-2
   q^4-1}{q^8},\frac{-2 q^{14}+q^{12}-2
   q^6-1}{q^8},-q^8-\frac{1}{q^8}-2,\\
   \frac{-q^{14}-2
   q^{12}+q^6-q^4-1}{q^6},\frac{-q^{14}-2 q^{10}+q^4-q^2-1}{q^6},\frac{-q^{14}-2
   q^8+q^2-2}{q^6},-q^8-\frac{1}{q^8}-2\Big\}.
\end{multline*}
The $8\times 8$ matrix is written with respect to the ordered basis $$(v_1\otimes v_8, v_2\otimes v_7, v_3\otimes v_6, v_4\otimes v_5, v_8\otimes v_1, v_7\otimes v_2, v_6\otimes v_3,v_5\otimes v_4)$$ It has rank $5$, with three linearly independent eigenvectors of eigenvalue $0$:
$$
\left\{0,-q^2,q,0,0,{-1},q,0\right\}^T, \quad 
\left\{0,0,-q^2,q,0,0,-1,q\right\}^T, \quad
 \left\{-q^2,q,0,0,-1,q,0,0\right\}^T.
$$
Again, note that the three vectors have the same nonzero entries as each other, and as $g_1$ and $g_2$. This will imply that all five processes have the same reversible measures.

\paragraph{Third eigenvector}

After conjugating and removing the third, fourth, seventh and eighth rows and columns, we get the generator
\begin{equation}
\begin{pmatrix}
* & -\frac{-q^8+q^6-2}{q^6} & \frac{q^6+q^{-6}-2}{q^2} &q^6-q^4+2q^{-2} \\
-\frac{-q^8+q^6-2}{q^4} & * & -\frac{-2q^8+q^2-1}{q^8} & q^6+q^{-6}-2\\
q^2(q^6+q^{-6}-2) & q^{-6}-q^{-4}+2q^{2} & * &-\frac{-2q^8+q^2-1}{q^2} \\
q^8-q^6+2 & q^6+q^{-6}-2 &-\frac{-2q^8+q^2-1}{q^4}  & *
\end{pmatrix}
\end{equation}
Associate to each vector $v_1,v_2,v_7,v_8$ a particle configuration as follows:
$$
v_1=\begin{tikzpicture}[scale=1,baseline=-1mm]
\node[draw,fill=white] at (-3.5,0.25){$2$};
    \node[draw,fill=white] at (-3.5,-0.25){$1$};
\end{tikzpicture},
\quad
v_2=\begin{tikzpicture}[scale=1,baseline=-1mm]
\node[draw,fill=white] at (-4,-0.25){$2$};
    \node[circle,draw,fill=white] at (-4,0.25){};
\end{tikzpicture},
\quad
v_7=\begin{tikzpicture}[scale=1,baseline=-1mm]
\node[draw,fill=white] at (-4,-0.25){$1$};
    \node[circle,draw,fill=white] at (-4,0.25){};
\end{tikzpicture},
\quad
v_8=\begin{tikzpicture}[scale=1,baseline=-1mm]
 \node[circle,draw,fill=white] at (-4,-0.25){};
    \node[circle,draw,fill=white] at (-4,0.25){};
\end{tikzpicture}
$$
We get the type D ASEP with parameters $(3,2)$. 

Let $v_4$ be the vacuum state, then 
$$qF_1E_1F_2E_2F_4E_3(v_4\otimes v_4)=-qv_1\otimes v_8+v_2\otimes v_7-q^{-1}v_8\otimes v_1+v_2\otimes v_2.$$

\begin{proposition}
For any $M,K,N\geq 0$, 
$$
F_1^KE_1^MF_2^{N-M}E_2^MF_4^{N-M}E_3^M\vert \Omega_N \rangle = \sum_{\eta} G(\eta)\vert \eta\rangle,
$$
where the sum is over all particle configurations $\eta$ on $N$ sites with $M$ second class particles and $N-K$ first class particles, and 
$$
\vert G(\eta) \vert= Z_{N,M,K}^{-1} \prod_{x_1 \in A_1(\eta)} q^{-x_1} \prod_{x_2\in A_2(\eta)} q^{-x_2}
$$
for some normalization constant $Z_{N,M,K}$.
\end{proposition}
\begin{proof}
The proof is similar to before, so the details are omitted.
\end{proof}

Proposition \ref{rev} and Theorem \ref{dual} hold for this process.

\paragraph{First eigenvector}
After conjugating and removing the first, forth, fifth and eighth rows and columns, we get the generator
\begin{equation}
\begin{pmatrix}
* & -\frac{-q^8+q^6-2}{q^4} & \frac{q^6+q^{-6}-2}{q^2} &-\frac{-q^8+q^6-2}{q^2} \\
-\frac{-q^8+q^6-2}{q^2} & * & -\frac{-2q^8+q^2-1}{q^8} & q^6+q^{-6}-2\\
q^2(q^6+q^{-6}-2) & -\frac{-2q^8+q^2-1}{q^6} & * &-\frac{-2q^8+q^2-1}{q^4} \\
q^8-q^6+2 & q^6+q^{-6}-2 &-\frac{-2q^8+q^2-1}{q^6}  & *
\end{pmatrix}
\end{equation}
where the diagonal entries are the unique expressions such that the rows sum to $0$. Associate to each vector $v_2,v_3,v_6,v_7$ a particle configuration as follows:
$$
v_2=\begin{tikzpicture}[scale=1,baseline=-1mm]
\node[draw,fill=white] at (-3.5,0.25){$2$};
    \node[draw,fill=white] at (-3.5,-0.25){$1$};
\end{tikzpicture},
\quad
v_3=\begin{tikzpicture}[scale=1,baseline=-1mm]
\node[draw,fill=white] at (-4,-0.25){$2$};
    \node[circle,draw,fill=white] at (-4,0.25){};
\end{tikzpicture},
\quad
v_6=\begin{tikzpicture}[scale=1,baseline=-1mm]
\node[draw,fill=white] at (-4,-0.25){$1$};
    \node[circle,draw,fill=white] at (-4,0.25){};
\end{tikzpicture},
\quad
v_7=\begin{tikzpicture}[scale=1,baseline=-1mm]
 \node[circle,draw,fill=white] at (-4,-0.25){};
    \node[circle,draw,fill=white] at (-4,0.25){};
\end{tikzpicture}
$$One obtains the type $D$ ASEP with parameters $(3,1)$.

Then Proposition \ref{rev} and Theorem \ref{dual} hold for this process, using $F_2$ in place of $F_1$.

\paragraph{Second eigenvector}
After conjugating and removing the first, second, fifth and sixth rows and columns, we get the generator
\begin{equation}
\begin{pmatrix}
* & q^6-q^{4}+2q^{-2} & \frac{q^6+q^{-6}-2}{q^2} & q^6-q^{4}+2q^{-2} \\
q^8-q^{6}+2 & * & -\frac{-2q^8+q^2-1}{q^8} & q^6+q^{-6}-2\\
q^2(q^6+q^{-6}-2) & q^{-6}-q^{-4}+2q^{2} & * & q^{-6}-q^{-4}+2q^{2} \\
q^8-q^6+2 & q^6+q^{-6}-2 &-\frac{-2q^8+q^2-1}{q^8}  & *
\end{pmatrix}
\end{equation}
Associate to each vector $v_3,v_4,v_5,v_6$ a particle configuration as follows:
$$
v_3=\begin{tikzpicture}[scale=1,baseline=-1mm]
\node[draw,fill=white] at (-3.5,0.25){$2$};
    \node[draw,fill=white] at (-3.5,-0.25){$1$};
\end{tikzpicture},
\quad
v_4=\begin{tikzpicture}[scale=1,baseline=-1mm]
\node[draw,fill=white] at (-4,-0.25){$2$};
    \node[circle,draw,fill=white] at (-4,0.25){};
\end{tikzpicture},
\quad
v_5=\begin{tikzpicture}[scale=1,baseline=-1mm]
\node[draw,fill=white] at (-4,-0.25){$1$};
    \node[circle,draw,fill=white] at (-4,0.25){};
\end{tikzpicture},
\quad
v_6=\begin{tikzpicture}[scale=1,baseline=-1mm]
 \node[circle,draw,fill=white] at (-4,-0.25){};
    \node[circle,draw,fill=white] at (-4,0.25){};
\end{tikzpicture}
$$
One obtains the type $D$ ASEP with parameters $(3,0)$. Then Proposition \ref{rev} and Theorem \ref{Thm2} hold for this process, using $F_3$ and $F_4$ in place of $F_2$ and $F_3$.

\section{An alternative construction}\label{Third}
To construct interacting particle systems from the Casimir of $\mathcal{U}_q(\mathfrak{so}_{6})$ and $\mathcal{U}_q(\mathfrak{so}_8)$, certain states were removed from the state space. This section provides an alternative method to construct interacting particle systems from the Casimir $\Omega$ of $\mathcal{U}(\mathfrak{so}_{2n})$, which does not require the removal of undesired states. The key difference is to subtract a diagonal matrix from $\Delta(\Omega)$, rather than subtracting a multiple of the identity matrix.

\subsection{Analyzing the Generator Matrix from \texorpdfstring{\sotwon}{so2n}}

\subsubsection{Obtaining \texorpdfstring{$G_n$}{Gn}}

In this section, we describe the process of arriving at $G_n$, a generator matrix of a Markov process, from the Casimir element of a Lie algebra of the form \sotwon \ for $n\geq 2$.

First, for this section, we redefine $H_i$ to be $E_{i,i}-E_{n+i,n+i}$. Note that now the Cartan-Weyl basis and corresponding dual basis for any \sotwon \ follows the same form:

For $i\leq n$
\begin{itemize}
    \item $H_i$ has dual $\frac{1}{4n-4}H_i$.
\end{itemize}
For $i<j\leq n$:
\begin{itemize}
    \item $X_{ij}$ has dual $\frac{1}{4n-4}X_{ji}$.
    \item $X_{ji}$ has dual $\frac{1}{4n-4}X_{ij}$.
    \item $Y_{ij}$ has dual $-\frac{1}{4n-4}Z_{ij}$.
    \item $Z_{ij}$ has dual $-\frac{1}{4n-4}Y_{ij}$.
\end{itemize}
These can be calculated based on the Killing Form ($H_i$ and $X_{ij}$ products have trace 2, and $Y_{ij}$ and $Z_{ij}$ products have trace $-2$).

The Casimir element is $\Omega=\sum\limits_iA_iA^i$ where each $A_i$ is a unique element of the Cartan-Weyl basis and $A^i$ is its counterpart in the dual basis. Let $\mathbb{C}^{2n}$ be the usual vector representation on \sotwon, so $\rhoCC$ defines a representation. $\rho(\Omega)=\sum\limits_i\rhoCC(A_i)\rhoCC(A^i)$ thus defines a $4n^2\times4n^2$ matrix.

Let $C$ be a $4n^2\times4n^2$ diagonal matrix with entries $c_i$ such that $c_i=\sum\limits_j\rho(\Omega)_{ij}$. Let $H=\rho_{\mathbb{C}^{2n}\otimes\mathbb{C}^{2n}}(\Omega)-C$. Let $G_n$ be the matrix found by negating rows of $\rho(\Omega)-C$ to force each diagonal entry to be non-positive. This leads to a generator matrix with all entries finite.

Note: Previous research on similar Lie algebras has used $C=kI$ for scalar $k$ and then used conjugation to arrive at a generator matrix \cite{KuanJPhysA}. By allowing the diagonal matrix $C$ to have different diagonal entries, we can construct an interacting particle system without discarding states. 

\subsubsection{Properties of \texorpdfstring{$G_n$}{Gn}}

In this section, we outline properties of $G_n$ that will be helpful in analyzing the states of the Markov process $G_n$ describes.

\begin{lemma}\label{lem:block}
The representation of the Casimir element $\rho(\Omega)$, a $4n^2 \times 4n^2$ matrix, can be written as a $2n \times 2n$ matrix of blocks of size $2n \times 2n$ in block form as:
\[
\rho(\Omega)=\frac{1}{2n-2}\begin{pmatrix}
\begin{matrix}
D_1 & X_{21} & X_{31} & \cdots & X_{n,1}\\
X_{12} & D_2 & X_{32} & \cdots & X_{n,2}\\
X_{13} & X_{23} & D_3 & \cdots & X_{n,3}\\
\cdots & \cdots & \cdots & \cdots & \cdots\\
X_{1,n} & X_{2,n} & X_{3,n} & \cdots & D_n
\end{matrix} & \vline & \begin{matrix}
0 & -Z_{12} & -Z_{13} & \cdots & -Z_{1,n}\\
Z_{12} & 0 & -Z_{23} & \cdots & -Z_{2,n}\\
Z_{13} & Z_{23} & 0 & \cdots & -Z_{3,n}\\
\cdots & \cdots & \cdots & \cdots & \cdots\\
Z_{1,n} & Z_{2,n} & Z_{3,n} & \cdots & 0
\end{matrix}\\
\hline
\begin{matrix}
0 & -Y_{12} & -Y_{13} & \cdots & -Y_{1,n}\\
Y_{12} & 0 & -Y_{23} & \cdots & -Y_{2,n}\\
Y_{13} & Y_{23} & 0 & \cdots & -Y_{3,n}\\
\cdots & \cdots & \cdots & \cdots & \cdots\\
Y_{1,n} & Y_{2,n} & Y_{3,n} & \cdots & 0
\end{matrix} & \vline & \begin{matrix}
D_{n+1} & -X_{12} & -X_{13} & \cdots & -X_{1,n}\\
-X_{21} & D_{n+2} & -X_{23} & \cdots & -X_{2,n}\\
-X_{31} & -X_{32} & D_{n+3} & \cdots & -X_{3,n}\\
\cdots & \cdots & \cdots & \cdots & \cdots\\
-X_{n,1} & -X_{n,2} & -X_{n,3} & \cdots & D_{2n}
\end{matrix}
\end{pmatrix}
\]
where $D_i=(2n-1)I+H_i$ and $D_{n+i}=(2n-1)I-H_i$ for $i\leq n$.
\end{lemma}
\begin{proof}[Proof]

We can partition the Cartan Weyl basis for \sotwon \ into disjoint subsets according to pairs $(i,j)$ for $i<j\leq n$, keeping the Cartan subalgebra generators ($H_i$ matrices) separate. Each of the subsets will, for a pair $(i,j)$, contain $X_{ij}$, $X_{ji}, Y_{ij}, Z_{ij}$.

Suppose we take an arbitrary part of the partitioned basis, i.e. choose $(i,j)$ such that $i<j\leq n$. We will analyze the contribution this part (specifically, the $X_{ij}$, $X_{ji}$, $Y_{ij}$, and $Z_{ij}$) makes to the Casimir representation.

Each of these matrices can have their representation written in block form, which places the matrix on each of the diagonals and places $I$'s, the identity matrix, in the locations indicated by the matrix. For example, in \sofour:
\[
\rho_{\mathbb{C}^4\otimes\mathbb{C}^4}(X_{12})=X_{12}\otimes I_4 + I_4 \otimes X_{12} =\begin{pmatrix}
X_{12} & I & 0 & 0\\
0 & X_{12} & 0 & 0\\
0 & 0 & X_{12} & 0\\
0 & 0 & -I & X_{12}\\
\end{pmatrix}.
\]
Take $\rhoCC(X_{ij})$ and $\rhoCC(X_{ji})$ using this form. Ignoring the $\frac{1}{4n-4}$ scaling, we want to analyze $\rhoCC(X_{ij})\rhoCC(X_{ji})+\rhoCC(X_{ji})\rhoCC(X_{ij})$. In order to do so, we will analyze a $6\times 6$ block matrix which captures the contributions to the $i,j,k,n+i,n+j,n+k$ block rows and columns, where $k\leq n$ such that $k\neq i,j$.
Letting $P_{ij}=X_{ij}X_{ji}+X_{ji}X_{ij}$, this results in the following matrix:
\[
\rhoCC(X_{ij})\rhoCC(X_{ji})+\rhoCC(X_{ji})\rhoCC(X_{ij})=\begin{pmatrix}
P_{ij}+I & 2X_{ji} & 0 & 0 & 0 & 0\\
2X_{ij} & P_{ij}+I & 0 & 0 & 0 & 0\\
0 & 0 & P_{ij} & 0 & 0 & 0\\
0 & 0 & 0 & P_{ij}+I & -2X_{ij} & 0\\
0 & 0 & 0 & -2X_{ji} & P_{ij}+I & 0\\
0 & 0 & 0 & 0 & 0 & P_{ij}
\end{pmatrix}.
\]
Notice that $P_{ij}$ is a diagonal matrix with 1 on the $i,j,n+i,n+j$ diagonals and 0 elsewhere. The block form of the identity matrix also has an $I$ added to the $P_{ij}$ at the $i,j,n+i,n+j$ rows. This means that, in the full $\rho(\Omega)$ matrix, the $X$ terms contribute $2(n-1)I$ to each diagonal term. Notice also that $X_{ij}$ and $X_{ji}$ appear in blocks of $\rho(\Omega)$ at locations determined by $i,j,n+i,n+j$.

Likewise, take $\rhoCC(Y_{ij})$ and $\rhoCC(Z_{ij})$ using this form. Note that the coefficient of the dual basis elements will be negative ($\frac{-1}{4n-4}$), meaning these matrices will be negated before being included in the Casimir sum. Analyzing this matrix the same way as the $X$'s, and letting $R_{ij}$ be $Y_{ij}Z_{ij}+Z_{ij}Y_{ij}$, we get the result:
\[
\rhoCC(Y_{ij})\rhoCC(Z_{ij})+\rhoCC(Z_{ij})\rhoCC(Y_{ij})=\begin{pmatrix}
R_{ij}-I & 0 & 0 & 0 & 2Z_{ij} & 0\\
0 & R_{ij}-I & 0 & -2Z_{ij} & 0 & 0\\
0 & 0 & R_{ij} & 0 & 0 & 0\\
0 & 2Y_{ij} & 0 & R_{ij}-I & 0 & 0\\
-2Y_{ij} & 0 & 0 & 0 & R_{ij}-I & 0\\
0 & 0 & 0 & 0 & 0 & R_{ij}
\end{pmatrix}.
\]
Notice that $R_{ij}$ is a diagonal matrix with $-1$ on the $i,j,n+i,n+j$ diagonals and 0 elsewhere. The block form of the identity matrix also has an $I$ subtracted from $R_{ij}$ at the $i,j,n+i,n+j$ rows. This means that, in the full $\rho(\Omega)$ matrix, since this matrix is negated, the $Y,Z$ terms contribute $2(n-1)I$ to each diagonal term. Notice also that $Y_{ij}$ and $Z_{ji}$ appear in locations determined by $i,j,n+i,n+j$, and again note that these will be negated in the final result.

In summary, we have shown that from pairs of $(i,j)$ in the $X,Y,Z$ elements of the basis, we generate each off-diagonal block based on one root vector and an equal value of $4(n-1)I$ in each diagonal block.

To fully analyze a block form, we must also analyze the contributions of the $H$ matrices to $\rho(\Omega)$. However, these matrices are diagonal, so their representations will also be diagonal and they will only contribute to the diagonal blocks. We can define a block representation of $H_i$, which puts $H_i$ in each diagonal block, adds $I$ in the $i$-th diagonal block, and subtracts $I$ in the $n+i$-th diagonal block.

Define a $4\times 4$ matrix indexed by $i,j,n+i,n+j$ such that $j\neq i$ and $j\leq n$. The result is the following matrix (up to scaling):
\[
\rhoCC(H_i)\rhoCC(H_i)=\begin{pmatrix}
H_i^2+2H_i+I & 0 & 0 & 0\\
0 & H_i^2 & 0 & 0\\
0 & 0 & H_i^2-2H_i+I & 0\\
0 & 0 & 0 & H_i^2
\end{pmatrix}.
\]
Adding this up for all $i$ shows that this contributes differently to each diagonal. Let $d_i$ be the $i$-th diagonal block of $\sum\limits_i\rhoCC(H_i)\rhoCC(H_i)$. Then $d_i=2I+2H_i$ and $d_{n+i}=2I-2H_i$ for $i\leq n$.

When we put all of these pieces together, add the scaling back in (putting $\frac{1}{4n-4}$ on the outside of the matrix, as this was the coefficient in every case), and pull a constant of 2 out of the matrix (leading to $\frac{1}{2n-2}$ as the coefficient), we get the block form in the lemma.
\end{proof}

\begin{lemma}\label{lem:GnMarkov}
The $G_n$ matrix is the generator of a Markov process.
\end{lemma}
\begin{proof}[Proof]

By the procedure used to construct the $G_n$ from $\rho(\Omega)$, it immediately follows that all of the rows (in the normal sense) of the resulting matrix sum to 0 and their diagonals are non-positive. Therefore it suffices to show that all off-diagonal entries of are non-negative.

Using the block form of $\rho(\Omega)$, we see that the set of rows of $\rho(\Omega)-C$ with negative elements is given by $S=\{n+1,n+1+1(2n+1),n+1+2(2n+1),\dots,2n^2,2n^2+1,2n^2+1+1(2n+1),\dots,4n^2-n\}$. This constitutes the $n+1$ row of the 1st block, the $n+2$ row of the 2nd block., etc., up to the $2n$ row of the $n$-th block, and then the 1st row of the $n+1$-th block, the 2nd row of the $n+2$-th block, etc. We also notice that each of these rows is given a negative value by every block that is not on the diagonal or the 0 diagonals of the upper right and bottom left. This means that each of these rows contains $2n-2$ off diagonal negative elements, all of which will be $-1$, because they come from root vectors.

We want to show that the diagonal value in each row of $S$ is equal to $2n-2$, which would imply that these rows already sum to 0 and thus will solely be negated, not adjusted by a constant. This is true because, in the $D_i$'s, $2n-2$ will occur at the $n+1$ row, and then every $2n+1$ rows up to $2n^2$. In the $D_{n+i}$'s, $2n-2$ will occur at the $2n^2+1$ row and then every $2n+1$ rows after that up to $4n^2-n$. Thus, the rows indexed by $S$ will have a 0 on that diagonal of $C$ and will be the rows negated at the end.

Note that all other rows of $C$ will have non-negative diagonals and non-negative off-diagonals, which will force the diagonals to become negative when $C$ is subtracted and thus the off-diagonals will remain positive.

Hence, all off-diagonal entries are non-negative, and thus the matrix $G_n$ is a generator matrix of a Markov process.
\end{proof}

Because $G_n$ is a $4n^2 \times 4n^2$ matrix, we can see that $G_n$ will represent a Markov process with $4n^2$ states.

\begin{lemma}\label{lem:eqoffdiag}
All non-zero off-diagonal entries of $G_n$ are equal.
\end{lemma}
\begin{proof}[Proof]
It follows from Lemma \ref{lem:block} that all off-diagonal entries of $G_n$ occur in the off-diagonal blocks, which are determined solely by the root vectors. However, the root vectors contain entries of only $1,0,-1$. Since all off-diagonal elements are non-negative, this implies that all nonzero off-diagonal entries of the matrix are the same and equal to $\frac{1}{2n-2}$.
\end{proof}

This lemma is useful because it means that the diagonals will follow the same pattern as the rows in regards to what states they represent, and since all off-diagonal entries are equal, the diagonal of a row captures what kind of state it represents in a Markov process.

\subsubsection{Expected Properties of the Markov process from \texorpdfstring{$G_n$}{Gn}}

In this section, we will use the properties of $G_n$ to show what states the Markov process that $G_n$ describes would need to include.

\begin{lemma}\label{lem:GenAbsorb}
The Markov process with generator $G_n$ has $2n$ absorbing states.
\end{lemma}
\begin{proof}
First, note that an absorbing state is indicated in a generator matrix by a row where all entries are 0.

Notice from block form that the first row of the generator matrix has $0$'s for all off-diagonal elements (the first row of all $X_{j1}$ and $Z_{1j}$ is all 0's). This means that the first diagonal entry of $C$ will be equal to the first diagonal entry of $\rho(\Omega)$, which is $\frac{2n}{2n-2}$ and this will return an entire row of 0's.

Moreover, because every off-diagonal entry is positive and all nonzero off-diagonals are equal, this implies that every time the value $\frac{2n}{2n-2}$ recurs in the diagonal, it implies an absorbing state. By the structure of the $D$ matrices, this occurs once in every $D$ block, of which there are $2n$ total, implying $2n$ rows of all 0's. In a generator matrix, these types of rows imply absorbing states, showing that the Markov Process represented by this generator matrix has $2n$ absorbing states.
\end{proof}

\begin{definition}\label{def:maxrow}
We define a \textbf{maximal choice row} as a row in which no other rows have a greater number of nonzero off-diagonal elements. The set of all maximal choice rows is called the \textbf{maximal choice set}.
\end{definition}

\begin{lemma}\label{GenMaxrow}
The generator $G_n$ has $2n$ maximal choice rows, each of which have $2n-2$ elements.
\end{lemma}
\begin{proof}[Proof]
Let $S=\{n+1,n+1+1(2n+1),n+1+2(2n+1),\dots,2n^2,2n^2+1,2n^2+1+1(2n+1),\dots,4n^2-n\}$. We want to show $S$ is the maximal choice set. Recall from Lemma \ref{lem:GnMarkov} that each row of $S$ has $2n-2$ off-diagonal elements. Because each block of block form, as given by a root vector, has at most 1 entry in a row, and each block row has a 0 block, $2n-2$ is the maximal number of choices a row can have. This means that every element of $S$ is a maximal choice row.

Moreover, the structure of the block form matrix shows that every other row has less than $2n-2$ off-diagonal entries, implying that $S$ is the maximal choice set. Because $S$ has cardinality $2n$, there are $2n$ maximal choice rows.

By a similar process to finding the rows of $G_n$ with negative elements, thus populating $S$, and taking advantage of symmetry within the matrix, we can see that negative off-diagonals also only occur in columns indicated by elements of $S$. This implies that rows of $S$ only communicate to other rows of $S$, making $S$ a communicating class. Moreover, since each root vector has only one negative off-diagonal and using the structure of the block matrix, each row in $S$ fails to immediately reach a unique other row (based on the location of the 0 matrix), but all $2n-2$ other states in $S$ can be reached.
\end{proof}

\begin{definition}\label{def:pairrow}
We define a \textbf{pairwise row} as a row with exactly one nonzero off-diagonal entry. If pairwise row $r$ has its nonzero off-diagonal entry at column $s$, and $s$ is a pairwise row with nonzero off-diagonal entry at column $r$, rows $r$ and $s$ are called \textbf{pairwise states}.
\end{definition}

\begin{lemma}\label{lem:elsepair}
Any row in $G_n$ is either an absorbing state, or a maximal choice row, or a pairwise state.
\end{lemma}
\begin{proof}
Let $r$ be a row of $G_n$ that is not an absorbing state or a maximal choice row. Because $r$ is not an absorbing state, it must have at least one off-diagonal element. All of the negative off-diagonal elements of $\rho(\Omega)$ are used in maximal choice rows, and thus we wish to study the positive off-diagonal elements of $\rho(\Omega)$.

First, we show that $r$ is a pairwise row. If $r$ is in the top half of $G_n$, $r$ is in a block row $s$ made up of three types of matrices: $X_{ks}$ for $k\leq n$, $Z_{ts}$ for $t<s$, and $-Z_{sq}$ for $q>s$. Each $X_{ks}$ matrix will have its positive element in row $k$, each $Z_{ts}$ will have its positive element in row $n+t$, and each $-Z_{sq}$ will have its positive element in row $n+q$. This forms a partition of all of the non-absorbing and non-maximal rows in the top half of the matrix, so $r$ only contains one off-diagonal positive element. Similarly, we can show that if $r$ is in the bottom half of the block form matrix, $r$ only contains one off-diagonal positive element. Thus, $r$ is a pairwise row.

Next, we know that for each $r$, its one off-diagonal element is in a block determined by a root vector. We will show that this root vector implies a corresponding pairwise row $r'$ that makes $r,r'$ pairwise states. This can be done in 4 cases:

\begin{enumerate}
    \item If $r$ is a row that has its positive entry in the block $X_{ji}$, then it has positive entry at matrix coordinates $(j+(i-1)2n,i+(j-1)2n)$. The block $X_{ij}$ has its positive entry at matrix coordinates $(i+(j-1)2n,j+(i-1)2n)$. If $r'$ is the row that has its positive entry in block $X_{ij}$, then $r$ and $r'$ are pairwise states.
    \item If $r$ is a row that has its positive entry in block $-X_{ij}$, then it has positive entry at matrix coordinates $(2n^2+(i-1)2n+n+i,2n^2+(j-1)2n+n+j)$. The block $-X_{ji}$ has its positive entry at matrix coordinates $(2n^2+(j-1)2n+n+j,2n^2+(i-1)2n+n+i)$. If $r'$ is the row that has its pairwise entry in block $-X_{ji}$, then $r$ and $r'$ are pairwise states.
    \item If $r$ is a row that has its positive entry in block $Y_{ij}$, then it has positive entry at matrix coordinates $(2n^2+(j-1)2n+i,(i-1)2n+n+j)$. The block $-Z_{ij}$ has its positive entry at matrix coordinates $((i-1)2n+n+j,2n^2+(j-1)2n+i)$. If $r'$ is the row that has its pairwise entry in block $-Z_{ij}$, then $r$ and $r'$ are pairwise states.
    \item If $r$ is a row that has its positive entry in block $Z_{ij}$, then it has its positive entry at matrix coordinates $((j-1)2n+n+i,2n^2+(i-1)2n+j)$. $-Y_{ij}$ as its positive entry at matrix coordinates $(2n^2+(i-1)2n+j,(j-1)2n+n+i)$. If $r'$ is the row that has its pairwise entry in $-Y_{ij}$, then $r$ and $r'$ are pairwise states.
\end{enumerate}
Hence, we see that every row that is not absorbing or maximal choice has exactly one positive off-diagonal and communicates to and from exactly one other row.
\end{proof}

We have seen that the rows of $G_n$ imply a Markov Process that can be split up into completely independent communicating classes, i.e. no state in a communicating class can reach a state in another communicating class. We have $2n$ communicating classes with 1 row (absorbing states), 1 communicating class with $2n$ rows such that each row can reach all others but 1 uniquely (maximum choice rows), and $2n^2-2n$ communicating classes of pairwise states, totaling $4n^2$ states. We will now turn our focus to the Type-$m$ Parallel SSEP and show that these criteria apply to this particle system.

\subsection{The Particle System}

\begin{definition}\label{def:parallel}
A \textbf{Parallel SSEP with N sites} is a system that has two separate $1$-dimensional SSEPs with $N\ge2$ sites each. Each site can either be empty or have a particle on it, and while particles can interact with neighboring sites on the same lattice, they cannot jump to the other lattice. A Parallel SSEP with 2 sites will be referred to as a \textbf{Basic Parallel SSEP}.
\end{definition}

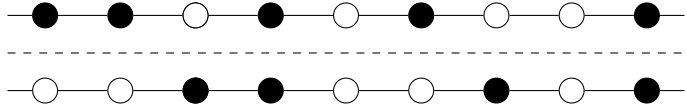
\begin{figure}[ht]
    \centering
    \begin{tikzpicture}[scale=1]
    \node[circle,draw,fill=black](leftleftpart) at (-4,0){};
    \node[circle,draw,fill=black](rightrightpart) at (4,0){};
    \draw (-4.5,0)--(leftleftpart);
    \draw (rightrightpart)--(4.5,0);
    \node[circle,draw,fill=black](leftpart) at (-3,0){};
    \draw (leftleftpart)--(leftpart);
    \node[circle,draw,fill=black](rightpart) at (-2,0){};
    \draw (leftpart)--(rightpart);
    \node[circle,draw,fill=white](leftpart) at (-2,0){};
    \draw (rightpart)--(leftpart);
    \node[circle,draw,fill=black](rightpart) at (-1,0){};
    \draw (leftpart)--(rightpart);
    \node[circle,draw,fill=white](leftpart) at (0,0){};
    \draw (rightpart)--(leftpart);
    \node[circle,draw,fill=black](rightpart) at (1,0){};
    \draw (leftpart)--(rightpart);
    \node[circle,draw,fill=white](leftpart) at (2,0){};
    \draw (rightpart)--(leftpart);
    \node[circle,draw,fill=white](rightpart) at (3,0){};
    \draw (leftpart)--(rightpart);
    \draw (rightpart)--(rightrightpart);
    \draw[dashed] (-4.5,-0.5)--(4.5,-0.5);
    \node[circle,draw,fill=white](leftleftpart) at (-4,-1){};
    \node[circle,draw,fill=black](rightrightpart) at (4,-1){};
    \draw (-4.5,-1)--(leftleftpart);
    \draw (rightrightpart)--(4.5,-1);
    \node[circle,draw,fill=white](leftpart) at (-3,-1){};
    \draw (leftleftpart)--(leftpart);
    \node[circle,draw,fill=white](rightpart) at (-2,-1){};
    \draw (leftpart)--(rightpart);
    \node[circle,draw,fill=black](leftpart) at (-2,-1){};
    \draw (rightpart)--(leftpart);
    \node[circle,draw,fill=black](rightpart) at (-1,-1){};
    \draw (leftpart)--(rightpart);
    \node[circle,draw,fill=white](leftpart) at (0,-1){};
    \draw (rightpart)--(leftpart);
    \node[circle,draw,fill=white](rightpart) at (1,-1){};
    \draw (leftpart)--(rightpart);
    \node[circle,draw,fill=black](leftpart) at (2,-1){};
    \draw (rightpart)--(leftpart);
    \node[circle,draw,fill=white](rightpart) at (3,-1){};
    \draw (leftpart)--(rightpart);
    \draw (rightpart)--(rightrightpart);
    \end{tikzpicture}
    \caption{A possible configuration of a Parallel SSEP with $9$ sites}
    \label{Parallel SSEP}
\end{figure}
\begin{definition}\label{def:TypemBasic}
A \textbf{Type-m Basic Parallel SSEP} is a Basic Parallel SSEP where the lower lattice allows particles of mass $1$, the upper lattice allows particles of mass $\omega\in\{\frac{1}{m},\frac{2}{m},\dots,\frac{m-1}{m},1\}$, and the following properties hold:
\begin{itemize}
    \item \textbf{Mass Order Property}: A particle can only move if no lighter particles can move.
    \item \textbf{Balance Property}: A set of
    balanced states exists in which the two lattices each have mass $1$ and particles are able to undergo fusion and fission, defined respectively as donating mass to or taking mass from a neighboring site. These mass-preserving processes allow a concurrent move in the lower lattice.
    \item \textbf{Class Property}: A particle in a non-balanced state is able to switch places with neighboring particles of higher mass.
\end{itemize}
\end{definition}

\begin{definition}\label{def:maxstates}
We define \textbf{maximum-choice states} as the states such that no other possible states in the system have more \textbf{choices}, states that can be immediately reached. In a Type-m Basic Parallel SSEP, the maximum-choice states are called \textbf{cyclic states} and form a communicating class, called a \textbf{cycle}.
\end{definition}

\begin{lemma}\label{lem:typembasicstates}
A Type-$m$ Basic Parallel SSEP has $4(m+1)^2$ states.
\end{lemma}
\begin{proof}[Proof]
By the definition of Type-$m$ Basic Parallel SSEP, the particle system contains 2 parallel latices with two sites each. The top lattice consists of 2 sites with $m+1$ configurations each ($m$ types of particles or empty). Thus, the top lattice has $(m+1)^2$ configurations.

Similarly, because the bottom lattice consists of 2 sites with 2 states each, there are 4 possible configurations of the bottom lattice. Since all possible combinations of a configuration of the top lattice and a configuration of the bottom lattice are valid states, there are $4(m+1)^2$ states.
\end{proof}

\begin{lemma}\label{lem:typembasicabsorb}
A Type-$m$ Basic Parallel SSEP has $2(m+1)$ absorbing states.
\end{lemma}
\begin{proof}
Absorbing states are reached by having each lattice occupied by either a pair of identical particles or a pair of empty sites. There are $m+1$ ways to accomplish this on the top lattice, and independently $2$ ways to accomplish this on the bottom lattice. Thus, there are $2(m+1)$ absorbing states.
\end{proof}

\begin{lemma}\label{lem:typembasiccyclic}
A Type-$m$ Basic Parallel SSEP has $2(m+1)$ cyclic states.
\end{lemma}
\begin{proof}
First we will prove that cyclic states only occur when the top and bottom lattice are equal at mass $1$ each. Begin by noting that if the lattices are equal at mass $1$, the state necessarily has more than $1$ choice. Assume then that the top and bottom lattice are not equal, for if they are equal at mass $2$ we are in an absorbing state. Then, if there is a movable particle of lowest mass it will only have $1$ choice of movement, thus the entire state has $1$ choice. Since cyclic states are our maximum-choice states, we then know that non-balanced states can not be cyclic.

Thus, cyclic states occur only when the mass of the top lattice and the mass of the bottom lattice are equal at $1$. This requires the two sites on the top lattice to have mass summing to $1$, and there are $m+1$ ways to arrange this. Independently, one particle of mass $1$ must be on one of the $2$ sites on the bottom lattice. Thus, there are $2(m+1)$ cyclic states.
\end{proof}

\begin{lemma}\label{lem:typembasicpair}
A state in the Type-$m$ Basic Parallel SSEP that is not absorbing or cyclic must be pairwise.
\end{lemma}
\begin{proof}[Proof]
Let $X$ be a state that is not absorbing or cyclic. $X$ cannot be the empty state, so it must have a movable particle of lowest mass. This particle will switch sites, resulting in the state $X'$. However, the movable particle of lowest mass from $X$ is still the movable particle of lowest mass in $X'$, and $X'$ will then switch back to $X$. Thus, $X$ is a pairwise state.
\end{proof}

\begin{theorem}\label{thm:basic}
Let $m=n-1$. The generator matrix of a Type-$m$ Basic Parallel SSEP is exactly $G_n$.
\end{theorem}
\begin{proof}
The previous lemmas demonstrate that the states implied in a Markov Process by the generator matrix, derived from a representation of a Casimir Element of \sotwon, are identical to the states in a Type-$m$ Basic Parallel SSEP. The $2n$ absorbing states in the matrix correspond to $2n$ absorbing states in the system, the $2n$ maximum choice rows (which form a communicating class) correspond to cyclic states (which are also a communicating class because mass is preserved), and everything else in both is split up into pairwise states. Hence, we conclude by state analysis that $G_n$ as a generator matrix represents Type-$m$ Basic Parallel SSEP.
\end{proof}

\subsubsection{An Example: \sofour}

The Cartan-Weyl basis for \sofour \ is $\{H_1,H_2,X_{12},X_{21},Y_{12},Z_{12}\}$. The Killing Form calculation implies that the dual basis for the Cartan-Weyl Basis is $\{\frac{1}{4}H_1,\frac{1}{4}H_2,\frac{1}{4}X_{21},\frac{1}{4}X_{12},-\frac{1}{4}Z_{12},-\frac{1}{4}Y_{12}\}$, corresponding to the order of the original Cartan-Weyl basis. Using the process described above, we arrive at a generator matrix for a 16-state Markov process.

The $16\times16$ generator matrix that results from \sofour is the following matrix:
\[
G=\begin{pmatrix}
\begin{smallmatrix}
0 & 0 & 0 & 0 & 0 & 0 & 0 & 0 & 0 & 0 & 0 & 0 & 0 & 0 & 0 & 0 \\
0 & -\half & 0 & 0 & \half & 0 & 0 & 0 & 0 & 0 & 0 & 0 & 0 & 0 & 0 & 0 \\
0 & 0 & -1 & 0 & 0 & 0 & 0 & \half & 0 & 0 & 0 & 0 & 0 & \half & 0 & 0 \\
0 & 0 & 0 & -\half & 0 & 0 & 0 & 0 & 0 & 0 & 0 & 0 & \half & 0 & 0 & 0 \\
0 & \half & 0 & 0 & -\half & 0 & 0 & 0 & 0 & 0 & 0 & 0 & 0 & 0 & 0 & 0 \\
0 & 0 & 0 & 0 & 0 & 0 & 0 & 0 & 0 & 0 & 0 & 0 & 0 & 0 & 0 & 0 \\
0 & 0 & 0 & 0 & 0 & 0 & -\half & 0 & 0 & \half & 0 & 0 & 0 & 0 & 0 & 0 \\
0 & 0 & \half & 0 & 0 & 0 & 0 & -1 & \half & 0 & 0 & 0 & 0 & 0 & 0 & 0 \\
0 & 0 & 0 & 0 & 0 & 0 & 0 & \half & -1 & 0 & 0 & 0 & 0 & \half & 0 & 0 \\
0 & 0 & 0 & 0 & 0 & 0 & \half & 0 & 0 & -\half & 0 & 0 & 0 & 0 & 0 & 0 \\
0 & 0 & 0 & 0 & 0 & 0 & 0 & 0 & 0 & 0 & 0 & 0 & 0 & 0 & 0 & 0 \\
0 & 0 & 0 & 0 & 0 & 0 & 0 & 0 & 0 & 0 & 0 & -\half & 0 & 0 & \half & 0 \\
0 & 0 & 0 & \half & 0 & 0 & 0 & 0 & 0 & 0 & 0 & 0 & -\half & 0 & 0 & 0 \\
0 & 0 & \half & 0 & 0 & 0 & 0 & 0 & \half & 0 & 0 & 0 & 0 & -1 & 0 & 0 \\
0 & 0 & 0 & 0 & 0 & 0 & 0 & 0 & 0 & 0 & 0 & \half & 0 & 0 & -\half & 0 \\
0 & 0 & 0 & 0 & 0 & 0 & 0 & 0 & 0 & 0 & 0 & 0 & 0 & 0 & 0 & 0 \\
\end{smallmatrix}
\end{pmatrix}
\]
As the generator of a Markov process, this matrix implies that there are 4 absorbing states (Rows 1, 6, 11,and 16), 4 states in a cycle (Rows 3, 8, 9, and 14). The other states are divided into pairs that switch back and forth (Rows 2 \& 5, 4 \& 13, 7 \& 10, and 12 \& 15). Note that $G$ can be rewritten as a block matrix based on the communicating classes of these states.

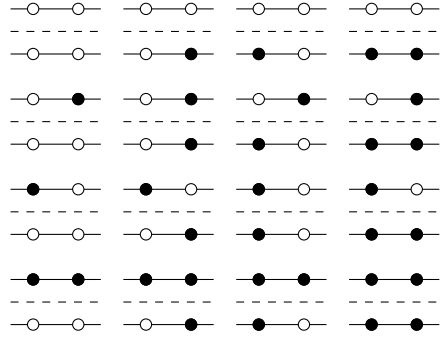
\begin{figure}[ht]
    \centering
    \begin{tikzpicture}[scale=.3]
    \foreach \x in {1,2,3,4} {
        \foreach \y in {1,2,3,4} {
        \draw[fill=white] (-1+5*\x,1+4*\y) circle[radius=.25]{};
        \draw[fill=white] (1+5*\x,1+4*\y) circle[radius =.25]{};
        \draw[-] (-2+5*\x,1+4*\y)--(-1.25+5*\x,1+4*\y);
        \draw[-] (-.75+5*\x,1+4*\y)--(.75+5*\x,1+4*\y);
        \draw[-] (1.25+5*\x,1+4*\y)--(2+5*\x,1+4*\y);
        \draw[dashed](-2+5*\x,0+4*\y)--(2+5*\x,0+4*\y);
        \draw[fill=white] (-1+5*\x,-1+4*\y) circle[radius=.25]{};
        \draw[fill=white] (1+5*\x,-1+4*\y) circle[radius =.25]{};
        \draw[-] (-2+5*\x,-1+4*\y)--(-1.25+5*\x,-1+4*\y);
        \draw[-] (-.75+5*\x,-1+4*\y)--(.75+5*\x,-1+4*\y);
        \draw[-] (1.25+5*\x,-1+4*\y)--(2+5*\x,-1+4*\y);}}
    \foreach \x in {1,2,3,4} {
        \draw[fill=black] (-1+5*\x,5) circle[radius=.25]{};
        \draw[fill=black] (1+5*\x,5) circle[radius=.25]{};
        \draw[fill=black] (-1+5*\x,5+4) circle[radius=.25]{};
        \draw[fill=black] (1+5*\x,5+8) circle[radius=.25]{};}
    \foreach \y in {1,2,3,4} {
        \draw[fill=black] (11,-1+4*\y) circle[radius=.25]{};
        \draw[fill=black] (14,-1+4*\y) circle[radius=.25]{};
        \draw[fill=black] (19,-1+4*\y) circle[radius=.25]{};
        \draw[fill=black] (21,-1+4*\y) circle[radius=.25]{};}
    \end{tikzpicture}
    \caption{16 Type-1 Basic Parallel SSEP Configurations}
    \label{so4 Particle System}
\end{figure}

The full list of states of the Type-1 Basic Parallel SSEP can be seen in Figure \ref{so4 Particle System}. It is straightforward to show that this particle system is generated by the generator derived before, as this particle system has four absorbing states, four states in a communicating class that can each immediately switch to two of the others but cannot immediately get to one of them, and eight states separated into four pairs that only switch back and forth, the same states that are suggested by $G$. Also, when a state has multiple choices it transitions to any one of them at an equal rate, as shown in the generator.

\subsection{Expansion of the Particle System to $N$ Sites}

\begin{definition}\label{def:subsys}
A \textbf{subsystem} of a Parallel SSEP with $N$ sites is a Basic Parallel SSEP formed from two of the $N$ adjacent sites. A Parallel SSEP with $N$ sites is made up of $N-1$ non-disjoint subsystems which define the local properties of the entire system.
\end{definition}

\begin{definition}\label{def:TypemN}
A \textbf{Type-m Parallel SSEP with N sites} is a Parallel SSEP with N sites such that every subsystem is a Type-m Basic Parallel SSEP.
\end{definition}

\begin{lemma}\label{lem:subchoices}
Let $S$ be the set of states of a Type-$m$ Basic Parallel SSEP such that the left-top site is fixed in state $a$ and the left-bottom site is fixed in state $b$. There are then $2(m+1)$ states in $S$ where $1$ state is absorbing, $1$ state is cyclic, and $2m$ states that are pairwise.
\end{lemma}
\begin{proof}
Clearly, we have $2$ choices for the bottom-right site and $m+1$ choices for the top-right site giving us a total of $2(m+1)$ states. Now, note that we have one absorbing state: top-right in state $a$ and bottom-right in state $b$. We also have one cyclic state: top-right in state $1-a$ and bottom-right in state $0$ if $b=1$ or state $1$ if $b=0$. By Lemma \ref{lem:typembasicpair} the remaining states are pairwise and so we have $2m$ pairwise states.
\end{proof}

\begin{lemma}\label{lem:typemNmax}
A Type-$m$ Parallel SSEP with $N$ sites has $2(m+1)$ maximum-choice states
\end{lemma}
\begin{proof}
Since the total number of choices of the system is the sum of the choices of its subsystems, we know that our maximum-choice states are those such that every subsystem is cyclic. Let our first subsystem be cyclic, then by Lemma \ref{lem:subchoices} we know that there is only one way for the rest of the subsystems to be cyclic. Thus, the total number of maximum-choice states is simply the total number of cyclic states for a Type-$m$ Basic Parallel SSEP, which by Lemma \ref{lem:typembasiccyclic} is $2(m+1)$.
\end{proof}

\begin{lemma}\label{lem:typemNabsorb}
A Type-$m$ Parallel SSEP with $N$ sites has $2(m+1)$ absorbing states
\end{lemma}
\begin{proof}
The proof is equivalent to the proof of lemma \ref{lem:typemNmax}, with absorbing states substituted for maximum-choice and cyclic states.
\end{proof}

\begin{theorem}\label{thm:expansion}
Let $L$ be the generator matrix from Theorem \ref{thm:basic}, and let: 
\begin{equation}\label{eq:typemexpansion}
L_N=\sum_{i=1}^{N-1} \underbrace{I\otimes\dots\otimes I}_{i-1} \otimes L \otimes \underbrace{I\otimes\dots\otimes I}_{N-1-i}
\end{equation}
$L_N$ is then the generator of a Type-$m$ Parallel SSEP with $N$ sites.
\end{theorem}
\begin{proof}
To provide motivation for the proof, note by Lemma \ref{lem:subchoices} that if our first subsystem is in a given state, we then have $(2(m+1))^{N-2}$ possible states for the entire system. Since our first subsystem has $4(m+1)^2$ choices by Lemma \ref{lem:typembasicstates} we then get $(2(m+1))^N$ total states for our system as expected. This reflects the idea that we can obtain our system's states by starting with our choices for the first subsystem, and then working our way subsystem by subsystem across the rest of our system. This idea combined with the previously stated fact that the total choices for a state of our system is the sum of the choices of the subsystems gives us our ability to prove this theorem using \eqref{eq:typemexpansion}, the formula for $L_N$.

By Lemma \ref{lem:block}, our generator matrix $L$ has diagonal in the form of the multiset $\{0,\underbrace{-1, -1, \dots, -1}_{2m},-2m\}$ repeated $2(m+1)$ times. This can be interpreted as the $2(m+1)$ ways to choose the leftmost sites' states, where the multisets then represent the ways to finish the subsystem and their resulting choice-counts seen in Lemma \ref{lem:subchoices}. Now, we will show that the $i^{th}$ term in \eqref{eq:typemexpansion} is corresponding directly to the $i^{th}$ subsystem. We are first making $(2(m+1))^{i-1}$ diagonal copies of $L$ to reflect the $(2(m+1))^{i-1}$ possible states of the system to the left of the $i^{th}$ subsystem. We then take what we have so far and make $(2(m+1))^{N-1-i}$ diagonal copies of each of its diagonal elements to allow for the $(2(m+1))^{N-1-i}$ possible ways to choose the rest of the system given the $i^{th}$ and previous subsystems have been chosen, and we note that these paths are reflected by the multiset expected from Lemma \ref{lem:subchoices}. Thus, by summing up the terms of \eqref{eq:typemexpansion} we are essentially summing up the choices for each possible subsystem, and thus $L_N$ is indeed the generator of a Type-$m$ Parallel SSEP with $N$ sites.
\end{proof}

\section*{Appendix A: Python code}

The code below implements many of the calculations described in the paper, particularly in Section 2. Here is a summary of the main functions:
\begin{itemize}
\item The \texttt{pair} function takes in two lists of indices and outputs the result of the $q$-pairing, except without the $(q - q)^{-1}$ factors. (For example, an output of $[0, -2]$ corresponds to $\left(1 + \frac{1}{q^2}\right)$ times some power of $-(q - q^{-1})$. 
\item The \texttt{mat} and \texttt{dual} functions take a basis of $e'$s (and their respetcive $f'$s) as described in Section $3$, and they output $M$ and $M^{-1}$, respectively.
\item The \texttt{result} function computes a basis of $e'$s or $f'$s by enumerating a list of all possible permutations, using \texttt{reduce} to remove duplicates due to commuting elements, and using \texttt{perm} to confirm that the pairing matrix stays nonsingular.
\end{itemize}

The code runs down the left column completely before continuing in the right column.

\begin{tiny}
\begin{multicols}{2}
\begin{verbatim}
from sympy import *
import itertools
from collections import deque

n = 4
var('q')

def a(i, j): # Cartan matrix lookup
    if (i == j):
        return 2
    if (abs(i-j) == 1 and max(i, j) <= n-1):
        return -1
    if (i == n-2 and j == n or i == n and j == n-2):
        return -1
    return 0

def pair(list1, list2): 
    ret = [] 
    if(len(list1) != len(list2)):
        return []
    if(len(list1) == 1):
        if(list1[0] == list2[0]):
            return [0]
        else:
            return []
    first = list2[0]
    for i in range(len(list1)):
        if(list1[i] == first):
            inductive_list1 = list1.copy()
            inductive_list1.pop(i)
            inductive_list2 = list2.copy()
            inductive_list2.pop(0)
            ih = pair(inductive_list1, inductive_list2)
            prefactors = 0
            for j in range(i):
                prefactors += a(list1[j], list1[i])
            for j in range(len(ih)):
                ih[j] += prefactors
            ret.extend(ih)
    return ret

def mat(listofLists):
    M = zeros(len(listofLists))
    for i in range(len(listofLists)):
        for j in range(len(listofLists)):
            lst = pair(listofLists[i], listofLists[j])
            for k in lst:
                M[i, j] += q**k
    return M.applyfunc(simplify)

def dual(listofLists): # for computing specific dual elements
    M = zeros(len(listofLists))
    for i in range(len(listofLists)):
        for j in range(len(listofLists)):
            lst = pair(listofLists[i], listofLists[j])
            for k in lst:
                M[i, j] += q**k
    N= M.inv()
    return N.applyfunc(simplify)

def perm(tentlist): # given list of lists, removes linear dependence
    fin = []
    for i in range(len(tentlist)):
        fin1 = list(fin)
        fin1.append(tentlist[i])
        M = mat(fin1)
        if(M.det() != 0):
            fin.append(tentlist[i])
    return fin

def result(setofindices): # gives a basis of elements
    tentlist = list(set(itertools.permutations(setofindices)))
    for i in range(len(tentlist)):
        tentlist[i] = list(tentlist[i])
    tentlist = reduce(tentlist)
    return perm(tentlist)

def reduce(tentlist): # remove duplicates
    visited=[0 for i in range(len(tentlist))]
    adj = [[0 for i in range(len(tentlist))] for j in range(len(tentlist))]
    finlist = []
    for n in range(len(tentlist)):
        l = tentlist[n]
        for ind in range(len(l) - 1):
            if(a(l[ind], l[ind+1]) == 0):
                ledit = l[:]
                ledit[ind], ledit[ind+1] = ledit[ind+1], ledit[ind]
                m = tentlist.index(ledit)
                adj[m][n] = 1
                adj[n][m] = 1
    adjacents = [[] for i in range(len(tentlist))]
    for i in range(len(tentlist)):
        temp = []
        for j in range(len(tentlist)):
            if(adj[i][j] == 1):
                temp.append(j)
        adjacents[i] = temp
    for n in range(len(tentlist)):
        if(visited[n] == 0):
            finlist.append(tentlist[n])
            qu = deque([n])
            while(qu):
                curr = qu.popleft()
                for neigh in adjacents[curr]:
                    if(visited[neigh] == 0):
                        visited[neigh] = 1
                        qu.append(neigh)
    return finlist
\end{verbatim}
\end{multicols}
\end{tiny}

\bibliographystyle{alpha}
\bibliography{REU2020Paper}

\end{document}